\documentclass[11pt]{article}
\usepackage[utf8]{inputenc}
\usepackage[a4paper]{geometry}
\usepackage{amsmath,amsbsy,amsgen,amscd,amsthm,amsfonts,amssymb,mathrsfs}

\usepackage{natbib}
\setcitestyle{numbers,square,comma}

\usepackage{lipsum}
\usepackage{amsfonts}
\usepackage{graphicx}
\usepackage{epstopdf}
\usepackage{algorithmic}

\usepackage{scalerel}
\usepackage{bbm}

\usepackage[dvipsnames]{xcolor}
\usepackage{xr-hyper}
\usepackage[colorlinks=true, linkcolor=blue, breaklinks=true]{hyperref}

\usepackage{graphicx}  
\usepackage{caption}
\usepackage{subcaption}
\graphicspath{ {./images/} }

\usepackage{pgfplots}

\usepackage{titlesec}
\titleformat*{\section}{\large\bfseries}
\titleformat*{\subsection}{\normalsize\bfseries}
\titleformat*{\subsubsection}{\small\bfseries}

\theoremstyle{plain}
\newtheorem{theorem}{Theorem}[section]
\newtheorem{lemma}[theorem]{Lemma}
\newtheorem{proposition}[theorem]{Proposition}

\newtheorem{assumption}[theorem]{Assumption}

\theoremstyle{definition}

\theoremstyle{remark}
\newtheorem{remark}[theorem]{Remark}

\title{Gaussian Process Regression under Computational and Epistemic Misspecification}
\author{Daniel Sanz-Alonso and Ruiyi Yang} 
\date{University of Chicago and Princeton University}

\usepackage{amssymb}
\newcommand{\R}{\mathbb{R}}
\usepackage{bbm}
\usepackage{comment}

\newcommand{\ssubset}{\subset\joinrel\subset}

\newcommand{\RKHS}{\mathcal{H}_\Phi}
\newcommand{\RKHSN}{\mathcal{H}_{\Phi_{\scaleto{N}{2.5pt}}}}
\newcommand{\RKHSpsi}{\mathcal{H}_\Psi}

\newcommand{\Inter}{\mathcal{I}_{\Phi,X_n}}
\newcommand{\InterN}{\mathcal{I}_{\Phi_N,X_n}}
\newcommand{\Interpsi}{\mathcal{I}_{\Psi,X_n}}
\newcommand{\h}{h_{n,\Omega}}
\newcommand{\LinfO}{L^{\infty}(\Omega)}
\newcommand{\LtwoO}{L^2(\Omega)}
\newcommand{\q}{q_{n}}
\newcommand{\B}{\mathbb{B}}
\newcommand{\D}{\mathcal{D}}
\newcommand{\NB}{\operatorname{NB}}
\newcommand{\so}{s_{\scaleto{0}{2.5pt}}}
\newcommand{\sN}{s_{\scaleto{N}{2.5pt}}}
\newcommand{\nugg}{\lambda}
\newcommand{\smin}{\sigma_{\operatorname{min}}}
\newcommand{\dinf}{L^{\infty}(X_n)}
\newcommand{\dtwo}{L^2(X_n)}
\newcommand{\C}{\mathcal{C}}
\newcommand{\I}{\mathcal{I}}
\newcommand{\Z}{\mathbb{Z}}

\begin{document}

\maketitle

\begin{abstract}
Gaussian process regression is a classical kernel method for function estimation and data interpolation. In large data applications, computational costs can be reduced using low-rank or sparse approximations of the kernel. This paper investigates the  effect of such kernel approximations on the interpolation error.
We introduce a unified framework to analyze Gaussian process regression under important classes of computational misspecification: Karhunen-Lo\`eve expansions that result in low-rank kernel approximations, multiscale wavelet expansions that induce sparsity in the covariance matrix, and finite element representations that induce sparsity in the precision matrix. Our theory also accounts for epistemic misspecification in the choice of kernel parameters. 
\end{abstract}

\section{Introduction}

Gaussian process (GP) regression, also known as \emph{kriging}, is an important tool in computational mathematics \cite{wendland2004scattered}, statistics \cite{stein1999interpolation}, and machine learning \cite{williams2006gaussian} 
with applications in geostatistics \cite{matheron1963principles}, computer experiments \cite{santner2003design}, inverse problems \cite{teckentrup2020convergence}, probabilistic numerics \cite{hennig2015probabilistic}, and Bayesian optimization \cite{srinivas2010gaussian}. 
Given the value of a function $f$ at some points, the goal is to infer the value of $f$ at unobserved locations. 
In GP regression, $f$ is modeled as a sample path from a GP with a user-chosen covariance kernel, so that its conditional mean interpolates the observed function values and provides an emulator for $f$. 
The accuracy of this emulator is sensitive to the choice of covariance kernel, which should reflect properties of $f$ such as its smoothness. 
However, these properties are often unknown or hard to impose on the kernel, motivating the study of GP regression under kernel misspecification.

We distinguish between two types of misspecification: epistemic and computational. 
\emph{Epistemic misspecification} refers to an inadequate choice of kernel model or kernel parameters caused
by lack of knowledge on the function $f.$
For example,
choosing the smoothness parameter in the Mat\'ern covariance kernel can be  challenging if the smoothness of $f$ is unknown \cite{zhang2004inconsistent}, and an inadequate choice of this parameter can hinder the convergence rate of GP regression
\cite{wang2020prediction,tuo2020kriging,wang2022gaussian,teckentrup2020convergence}.
Even when the choice of kernel dovetails with the function $f,$ \emph{computational misspecification} arises if the kernel needs to be approximated for computational reasons.
For instance, kernels defined as infinite series must be truncated, leading to a low-rank approximation. 
Relatedly, sparse kernel approximations are essential to tackle large data applications where the cubic cost of inverting a dense covariance matrix becomes prohibitive.
In this direction,  \cite{lindgren2011explicit} 
constructs approximate Matérn GPs with sparse precision matrices
by using finite elements to represent the solution of a stochastic partial differential equation (SPDE). 
Other low-rank and sparsity-promoting techniques include covariance tapering \cite{furrer2006covariance}, wavelet approximations \cite{bolin2013comparison,owhadi2019operator}, Fourier methods \cite{greengard2022equispaced}, hierarchical matrices \cite{geoga2020scalable}, and extensions of the SPDE approach to graphs \cite{sanz2022spde}. 
Despite the computational speedup offered by these techniques, they only provide an approximation to the kernels used for modeling. 
How such computational misspecification affects GP regression is yet not fully understood.

This paper presents a unified analysis of GP regression that simultaneously addresses epistemic and computational misspecification. We study the approximation error of a function $f$ by its GP interpolant defined using  misspecified kernels. In addition to the standard interpolation error incurred in the well-specified case, there is an approximation error in the misspecified setting that reflects the distance between $f$ and the reproducing kernel Hilbert space (RKHS) of the misspecified kernel. Leveraging results from approximation theory and scattered data approximation, we illustrate our unified framework in three important classes of GP methods: Karhunen-Lo\`eve expansions that result in low-rank kernel approximations, multiscale wavelet expansions that induce sparsity in the covariance matrix, and finite element representations that induce sparsity in the precision matrix. Our analysis provides new guidelines for the choice of the complexity parameter (e.g. truncation level or finite element mesh size) relative to the number of observations. In addition, our framework recovers existing theory for epistemic misspecification of the smoothness parameter in Mat\'ern kernels. 

\vspace{-1pt}
\subsection{Comparison to related work}
The study of GP regression under epistemic misspecification originates in the spatial statistics literature \cite{stein1988asymptotically,stein1993simple} with a focus on asymptotically efficient estimation. In this spirit, \cite{putter2001effect} analyzes a sequence of misspecified kernels, as we will do in this paper. More recent works derive the convergence rate of GP interpolants towards the truth as the number of observations increases.
The line of work  \cite{wang2020prediction,tuo2020kriging,wang2022gaussian} establishes interpolation error bounds with Mat\'ern-type kernels when the smoothness parameter is misspecified. The paper
\cite{teckentrup2020convergence} extends the analysis to a sequence of estimated smoothness parameters, while \cite{wynne2021convergence} studies GP regression under noise corruption and likelihood misspecification.

Computational misspecification has received less attention. Multiscale kernels were analyzed in \cite{opfer2006tight,griebel2015multiscale} in the absence of epistemic misspecification. A Bayesian nonparametrics analysis of the effect of finite element representations on the posterior contraction rate can be found in \cite{sanz2022finite}. However, we are not aware of existing interpolation error analyses of GP regression with truncated Karhunen-Lo\`eve expansions or finite element representations. In contrast to \cite{opfer2006tight,griebel2015multiscale,sanz2022finite,teckentrup2020convergence,wang2020prediction,tuo2020kriging,wang2022gaussian,wynne2021convergence}, our paper simultaneously addresses epistemic and computational misspecification, and further covers a wider range of kernels encountered in practice. 
We focus on deterministic kernel approximations; for random feature models,
convergence rates in terms of the number of features and the number of observations were established in \cite{minh2016operator,lanthaler2023error,hashemi2023generalization}.

\vspace{-1pt}
\subsection{Outline}
Section \ref{sec:background} summarizes the necessary background on GP regression and formalizes the problem setting. Section \ref{sec:framework} introduces a general framework to analyze GP regression under computational and epistemic misspecification. Adopting this framework, we analyze several important classes of GP methods in Section \ref{sec:examples}. Section \ref{sec:conclusions} closes. Standard proofs and auxiliary results are deferred to the Appendices.

\vspace{-1pt}
\subsection{Notation}
For real numbers $a,b$, we denote $a\wedge b=\text{min}(a,b)$ and $a\vee b=\text{max}(a,b)$. 
The symbol $\lesssim$ will denote less than or equal
to up to a universal constant and similarly for $\gtrsim$. For real sequences $\{a_n\},\{b_n\}$, we 
write $a_n\asymp b_n$ if $a_n\lesssim b_n$ and $b_n\lesssim a_n$ for all $n$.
For  sets $E_1, E_2 \subset \R^d,$ we write $E_1 \ssubset E_2$ if $E_1$ is compactly contained in $E_2.$

\section{Preliminaries}\label{sec:background}

\subsection{Problem formulation}
Let $\D \subset \R^d$ and suppose we can observe the values of a deterministic function $f: \D \to \R$ at design points $X_n:=\{x_1, \ldots, x_n \} \subset \Omega$ on an experimental region $\Omega \subset \D.$ Given data
$Y := \bigl(f(x_1),\ldots,f(x_n)\bigr)^T$ 
and a positive definite covariance function (or kernel) $\Psi: \D \times \D \to \R,$
GP regression estimates the value of $f$ at unobserved locations and quantifies the uncertainty in these predictions \cite{williams2006gaussian}.
Specifically, setting a GP prior $\mathcal{GP}(0, \Psi)$ on $f,$ inference at $x \in \D$ is based on the posterior mean and posterior predictive variance, given by
\vspace{-3pt}
\begin{align}
    \Interpsi  f(x) &:=k_{\Psi}(x)^T K_{\Psi}^{-1}Y, \label{eq:kriging estimator} \\
    P_{\Psi,X_n}(x) &:= \Psi(x,x) - k_{\Psi}(x)^T K_{\Psi}^{-1} k_{\Psi}(x), \label{eq:variance}
\end{align}
where $k_{\Psi}(x) := \bigl(\Psi(x,x_1),\ldots,\Psi(x,x_n)\bigr)^T$ and the kernel matrix $K_{\Psi}:=[\Psi(x_i,x_j)]_{ij}$ will be assumed to be invertible. Notice that $\Interpsi$ interpolates the given data, that is,  $\Interpsi f(x_i) = f(x_i)$ for $x_i \in X_n.$ For this reason, $\Interpsi f$ is called the \emph{kriging interpolant}. Moreover, it follows from \eqref{eq:variance} that $P_{\Psi,X_n}(x_i)=0$ at observed locations. Finally, notice that $\Interpsi f$ is linear in the observations $Y;$ in fact, $\Interpsi f$ is the \emph{best linear predictor} in the sense of minimizing mean squared error \cite{stein1999interpolation}. 

When the kernel matrix $K_\Psi$ is ill-conditioned, it is standard practice to include a \emph{nugget term}, $\lambda I$, and consider 
\vspace{-3pt}
\begin{align}
    \Interpsi^\nugg f(x) &:=k_{\Psi}(x)^T (K_{\Psi}+\nugg I)^{-1}Y, \label{eq:kriging estimatorwithnugget} \\
    P_{\Psi,X_n}^\nugg(x) &:= \Psi(x,x) - k_{\Psi}(x)^T (K_{\Psi}+\nugg I)^{-1} k_{\Psi}(x), \label{eq:variancewithnugget}
\end{align}
where $\lambda \ge 0$ is a small tunable parameter. Notice that $\lambda = 0$ recovers the posterior mean and variance in \eqref{eq:kriging estimator} and \eqref{eq:variance}. The nugget term acts as a regularization and alleviates the computational instability arising when inverting an ill-conditioned kernel matrix. For convenience, we will still refer to $\Interpsi^\nugg  f$ as a kriging interpolant, although in general $\Interpsi^\nugg  f$ no longer interpolates the data when $\lambda >0$. 

This paper studies how the choice of kernel affects the error of kriging interpolants. 
We consider a misspecified setting where $f$, as a deterministic function, belongs to a certain function space such as a Sobolev space or the RKHS $\RKHS$ of a ground-truth kernel $\Phi : \D \times \D \to \R,$ but the user computes the kriging interpolant $\InterN^\lambda f$ with a kernel $\Phi_N : \D \times \D \to \R.$  
Here, $N$ may index ---as in \cite{teckentrup2020convergence}--- a sequence of plug-in parameter estimates under epistemic misspecification, or it may represent a complexity parameter controlling the cost of a kernel approximation under computational misspecification. 
In Section \ref{sec:framework} we introduce a unified framework to obtain error bounds for kriging interpolants.

The main result, Theorem \ref{thm:misspecifed kriging error}, has two parts. 
The first part outlines an existing approach to bound  $\|\InterN^\nugg f - f \|_{L^2(\Omega)}$ and $\|\InterN^\nugg f - f \|_{L^\infty(\Omega)}$ in terms of the number of observations $n$ and the complexity parameter $N,$ whereas the second part introduces a novel approach to obtain such bounds by leveraging a new stability analysis (Lemma \ref{lemma:stability of interpolation}) of the interpolation operator $\InterN^\lambda$.
Then, in Section \ref{sec:examples} we illustrate the application of this general framework in several important examples, as outlined in Table \ref{table:summary}.

\begin{table}[h]
    \centering
    \begin{tabular}{|c|c|c|c|c|c|}
    \hline
          Setting  & Misspecification & $N $& Section  &Part \\  \hline \hline
          Matérn kernel & Epistemic   & Plug-in estimate  & \ref{sec:ex:para miss}  & I \\ 
          Karhunen-Lo\`eve & Epist. \& comp.  &Truncation level & \ref{sec:KL} & I and  II \\
          Wavelet multiscale kernel &  Epist. \& comp. & Resolution level& \ref{sec:wavelets} &II \\
          Finite element kernel & Epist. \& comp.  & Mesh size & \ref{sec:FEM} & II \\ \hline
    \end{tabular}
    \caption{Summary of the examples studied in Section \ref{sec:examples}, along with the interpretation of the index $N$ and the part of Theorem \ref{thm:misspecifed kriging error} used in the analysis.}
    \label{table:summary}
\end{table}

Before we delve into the new theory, we introduce some assumptions on the experimental region and the design points.

\subsection{Experimental region and design points}
Throughout this paper, we adopt the following assumption on the domain $\D$ and the experimental region $\Omega$:
\begin{assumption}\label{assp:Omega}
$\Omega\subset \D \subset \mathbb{R}^d$ are two domains with Lipschitz boundary. In addition, $\Omega$ is bounded and satisfies an interior cone condition.
\end{assumption}
We consider the design points $X_n = \{x_1, \ldots, x_n\} \subset \Omega$ to be fixed. 
The error of kriging interpolants depends on how well the design points cover the experimental region $\Omega.$ Intuitively, we would like the design points to be uniformly spread on $\Omega$. 
To formalize this intuition, we introduce three quantities that will appear in our bounds: the \emph{fill distance} $\h,$  the \emph{separation radius} $\q,$ and the \emph{mesh ratio} $\rho_{n,\Omega}$ defined by  
\begin{align}\label{eq:h q rho}
    \h := \underset{x\in\Omega}{\operatorname{sup}} \,\operatorname{dist}(x,X_n),\qquad \q :=\underset{x\neq z\in X_n}{\operatorname{min}}\,\, \frac12|x-z|,\qquad \rho_{n,\Omega} := \frac{\h}{\q}.  
\end{align}
The fill distance, also known as maximin distance \cite{johnson1990minimax} or dispersion \cite{niederreiter1992random}, is the maximum distance that any point $x\in \Omega$ can be from its closest design point. The separation radius is half the smallest distance between any two distinct design points. The mesh ratio quantifies whether the design points are close to a uniform grid.

Fill distance and separation radius are decreasing functions of $n,$ and tend to zero in the large $n$ limit for space-filling designs. 
For any space-filling design, the fill distance tends to zero at rate at most $h_n \le c n^{-1/d},$ see \cite{niederreiter1992random}. In contrast, the mesh ratio $\rho_{n,\Omega}$ is bounded below by $1$ and is an increasing function of $n.$ Sequences of design points $X_n$ for which $\rho_{n,\Omega}$ is uniformly bounded in $n$ are called \emph{quasi-uniform}. For quasi-uniform points, it holds that $\h \asymp \q \asymp n^{-1/d}.$

\section{Unified theoretical framework}\label{sec:framework}

Our main result in this section, Theorem \ref{thm:misspecifed kriging error}, establishes error bounds for approximating a deterministic function $f$ by its kriging interpolants defined with a sequence of kernels $\Phi_N$ whose RKHSs do not necessarily contain $f$. 
The key idea is to introduce an approximate function $f_N\in \RKHSN$ and reduce the original kriging error to more tractable ones. 
In Section \ref{sec:examples}, our analysis of computational and epistemic misspecification builds on this unified framework and we will leverage results from approximation theory to find the appropriate $f_N$'s.

\subsection{Main Result}
For $s>\frac{d}{2}$, we recall that the Sobolev space $H^s(\R^d)$ is defined as 
\begin{align*}
    H^s(\R^d)=\Big\{f\in L^2(\R^d): \|f\|^2_{H^s(\R^d)}=\int_{\R^d} (1+|\xi|^2)^s|\widehat{f}(\xi)|^2d\xi <\infty\Big\},
\end{align*}
where $\widehat{f}$ denotes the Fourier transform of $f$. For $\D\subset \R^d$, we define 
\begin{align*}
    H^s(\D) = \Big\{f\in L^2(\D): \exists\, F\in H^s(\R^d) \text{ such that } F|_\D= f \Big\}
\end{align*}
with norm
\begin{align*}
    \|f\|_{H^s(\D)} = \operatorname{inf}\Big\{\|F\|_{H^s(\R^d)}:F\in H^s(\R^d) \text{ and } F|_\D= f\Big\}.
\end{align*}

\begin{theorem}\label{thm:misspecifed kriging error}
    Let $\Phi_N:\D\times \D\rightarrow \R$ be a covariance function whose RKHS satisfies  $\RKHSN  \subset H^{\sN} (\D)$ for $s_N>\frac{d}{2}$ with equivalent norms: there exists $A_N>1$ such that $A_N^{-1}\|g\|_{\RKHSN }\leq \|g\|_{H^{\sN} (\D)} \leq A_N \|g\|_{\RKHSN }$ for any $g\in \RKHSN $.
    If $K_{\Phi_N}\in \mathbb{R}^{n\times n}$ is invertible, then the following holds: 

\vspace{0.25cm}
{\sc Part I.} Let $\Phi:\D\times\D\rightarrow\R$ be a positive definite covariance function whose RKHS satisfies $\RKHS \subset H^{\so}(\D)$ for $s_0>\frac{d}{2}$ with equivalent norms.
Let $f\in\RKHS$. Suppose that $\RKHSN \subset \RKHS$ and that there exists $f_N\in \RKHSN$ with $f_N|_{X_n}\equiv f|_{X_n}$. Then, there exists $h_{\lfloor \so\rfloor,\lfloor \sN\rfloor,d,\Omega}$ such that if $\h\leq h_{\lfloor \so\rfloor,\lfloor \sN\rfloor,d,\Omega}$ 
\begin{align*}
    \|f - \InterN^\nugg f\|_{\LtwoO} &\leq C_{\so,\sN,d,\Omega}A_N^2 \Big[\big(\h^{\so} +\h^{\frac{d}{2}}\sqrt{\nugg}\big)\big(\|f\|_{H^{\so}(\D)}+\|f_N\|_{H^{\so}(\D)}\big)\\
    &\qquad \qquad + \big(\h^{\sN}+\h^{\frac{d}{2}}\sqrt{\nugg}\big)\|f_N\|_{H^{\sN}(\D)}\Big],\\
    \|f - \InterN^\nugg f\|_{\LinfO} &\leq C_{\so,\sN,d,\Omega}A_N^2\Big[ \big(\h^{\so-\frac{d}{2}}+\sqrt{\nugg}\big) \big(\|f\|_{H^{\so}(\D)}+\|f_N\|_{H^{\so}(\D)}\big)\\&\qquad\qquad + \big(\h^{\sN-\frac{d}{2}}+\sqrt{\nugg}\big)\|f_N\|_{H^{\sN}(\D)}\Big].
\end{align*} 
If $\nugg=0$, the constant $C_{\so,\sN,d,\Omega}$ can be improved to $C_{\lfloor\so\rfloor,\lfloor\sN\rfloor,d,\Omega}$.

\vspace{0.25cm}
{\sc Part II.} 
Let $f\in L^\infty(\D)$ with well-defined pointwise values. There exists $h_{\lfloor \sN\rfloor,d,\Omega}$ such that if $\h\leq h_{\lfloor \sN\rfloor,d,\Omega}$, then, for any $f_N\in \mathcal{H}_{\Phi_N}$, 
\begin{align*}
    \|f-\InterN^\nugg f\|_{L^2(\Omega)}&\leq C_{\sN,d,\Omega}A_N\sqrt{n}\bigg[\frac{\h^{\sN}}{\sqrt{\smin(K_{\Phi_N})+\nugg}}+\h^{\frac{d}{2}}\bigg] \|f_N-f\|_{L^{\infty}(\Omega)}\\  &\qquad\qquad  +C_{\sN,d,\Omega}A_N^2\big(\h^{\sN}+\h^{\frac{d}{2}}\sqrt{\nugg}\big)\|f_N\|_{H^{\sN}(\D)},\\
     \|f-\InterN^\nugg f\|_{L^{\infty}(\Omega)}&\leq C_{\sN,d,\Omega}A_N\sqrt{n}\bigg[\frac{\h^{\sN-\frac{d}{2}}}{\sqrt{\smin(K_{\Phi_N})+\lambda}}+1\bigg]\|f_N-f\|_{L^{\infty}(\Omega)}\\
    &\qquad \qquad + C_{\sN,d,\Omega}A_N^2\big(\h^{\sN-\frac{d}{2}}+\sqrt{\nugg}\big) \|f_N\|_{H^{\sN}(\D)}.
\end{align*}
If $\nugg=0$, the constant $C_{\sN,d,\Omega}$ can be improved to $C_{\lfloor\sN\rfloor,d,\Omega}$.

\end{theorem}

Before proving Theorem \ref{thm:misspecifed kriging error}, we make some remarks. 

\begin{remark}
Due to misspecification, the function $f$ does not necessarily belong to the RKHS $\RKHSN$ of the kriging kernel so that standard interpolation bounds do not apply. 
As mentioned, we then introduce an approximate function $f_N\in \RKHSN$ and decompose the error into terms that we can control. 
In part I, we rely on an approximation $f_N$ that interpolates $f$ over the design points $X_n$. 
Since the kriging interpolant $\Interpsi^\nugg f$ defined in \eqref{eq:kriging estimatorwithnugget} depends only on $f|_{X_n}$, we can rewrite $\Interpsi^\nugg f$ as $\Interpsi^\nugg f_N$ for both $\Psi=\Phi$ and $\Psi = \Phi_N,$ and then employ standard interpolation results. 
Such an idea has already been used to deal with the setting where the truth $f$ is rougher than functions in $\RKHSN$ \cite{narcowich2002scattered,narcowich2004scattered,narcowich2006sobolev,griebel2015multiscale,tuo2020kriging} (a.k.a. escaping from the native space).
Our part I framework can then be adopted to recover the existing results as we shall demonstrate through an example in Subsection \ref{sec:ex:para miss}.

However, finding such an interpolant is not always possible,  especially in the presence of computational misspecification.  In part II, we relax this requirement by allowing $f_N$ to be any function in $\RKHSN$ that approximates $f$ well in $L^\infty(\Omega)$ norm, leading to a different error decomposition. 
The resulting bound consists of two terms representing (i) how well functions in $\RKHSN$ can approximate $f$ and (ii) the standard error of approximate interpolation for $f_N$. 
By leveraging results from approximation theory, we can get a good control on the first error and as a result optimal error rates for approximating $f$ in certain cases, as we show in Section \ref{sec:examples}. 
$\hfill \square$
\end{remark}

\begin{remark}
The $N$-dependent constants $C_{\sN,d,\Omega}$ and $A_N$, which come respectively from the \emph{sampling theorem} (see Lemma \ref{lemma:sampling theorem} below) and the norm equivalence condition, need to be controlled for obtaining rates of convergence. 
For $C_{\sN,d,\Omega}$, this can be done by restricting the range of $s_N$. In particular, when $\nugg=0$ where the constant improves to $C_{\lfloor\sN\rfloor,d,\Omega}$, it suffices to require the $s_N$'s to stay in a bounded interval since only their integer parts affect the constant.

For $A_N$, similar restrictions on the kernel parameters can be imposed to deal with epistemic misspecification, as we demonstrate in Subsection \ref{sec:ex:para miss}. For computational misspecification where $N$ acts as a truncation level, we show by examples in Subsections \ref{sec:KL}, \ref{sec:wavelets}, and \ref{sec:FEM} that $A_N$ can also be uniformly controlled. 
$\hfill \square$
\end{remark}

\begin{remark}
The introduction of a nugget term can be seen as another layer of computational misspecification, which can help both the computation and the analysis. As pointed out in \cite{wendland2005approximate}, 
the nugget term regularizes the kriging estimator without affecting the error rate, as can be seen by setting $\nugg$ appropriately in Theorem \ref{thm:misspecifed kriging error}. 
Especially in our part II analysis, we can sidestep bounding the smallest eigenvalue $\smin(K_{\Phi_N})$ of the kernel matrix, which is itself a challenging task, and still obtain the same rate by setting $\nugg\asymp \h^{\sN-d/2}$. $\hfill \square$
\end{remark}

\begin{remark}
Lastly, Theorem \ref{thm:misspecifed kriging error} implies that the convergence order will be dictated by the smaller of the regularity of the truth $f$ and that of the RKHS $\RKHSN$, where in part II the regularity of $f$ comes in the approximation error $\|f_N-f\|_{L^\infty(\Omega)}$. 
When $f$ is rougher than functions in $\RKHSN$, this is in general optimal (see e.g. \cite[Table 1]{tuo2020kriging}). 
When $f$ is smoother, improved convergence rates may be obtained under stronger assumptions following ideas from \cite{schaback1999improved,sloan2023doubling}, 
for instance by requiring that the approximate function $f_N$ belongs to a certain ``nice'' subspace of $\RKHSN$.
This requirement would impose structural restrictions on $\RKHSN$ that may not be met under computational misspecification. We hence do not pursue further the investigation of improved convergence rates for smooth $f$ in this paper. 
$\hfill \square$
\end{remark}

\subsection{Proof of the Main Result}\label{sec:proof of main result}
The rest of this section contains the proof of Theorem \ref{thm:misspecifed kriging error}, which relies on Lemmas \ref{lemma:approx interpolation error} and \ref{lemma:stability of interpolation}. 
The proofs of these two lemmas make use of standard results in scattered data approximation that we review in Lemmas \ref{lemma:sampling theorem}, \ref{lemma:regularized LS}, and \ref{lemma:approx inter error on Xn}, whose proofs are included for completeness in Appendix \ref{sec:Appendix1}.

Recall that $\Omega\subset\D$ are two domains as in Assumption \ref{assp:Omega}.   
For function $f\in L^\infty(\D)$ with well-defined pointwise values, we denote
\begin{align*}
    \|f\|_{\dinf}:=\underset{i=1,\ldots,n}{\operatorname{max}} \,\, |f(x_i)|, \qquad  \|f\|_{\dtwo}:=\Big(\sum_{i=1}^n |f(x_i)|^2\Big)^{1/2}.
\end{align*}
The first lemma is a \emph{sampling theorem}, which controls the $L^2(\Omega)$ and $L^{\infty}(\Omega)$ norms of a function $u$ by its Sobolev norm and its $L^2(X_n)$ norm.
\begin{lemma}[{{\cite[Theorem 3.1]{arcangeli2012extension} and \cite[Theorem 2.12]{narcowich2005sobolev}}}]
\label{lemma:sampling theorem}
Let $s>\frac{d}{2}$. There exist two positive constants $h_{\lfloor s \rfloor,d,\Omega}$ and $C_{s,d,\Omega}$ such that if $X_n$ is a discrete set satisfying $\h\leq h_{\lfloor s \rfloor,d,\Omega}$ and $u\in H^s(\Omega)$, then 
\begin{align*}
    \|u\|_{L^2(\Omega)}& \leq C_{s,d,\Omega}\left[\h^s\|u\|_{H^s(\Omega)}+\h^{\frac{d}{2}}\|u\|_{\dtwo}\right],\\
    \|u\|_{L^{\infty}(\Omega)}& \leq C_{s,d,\Omega} \left[\h^{s-\frac{d}{2}}\|u\|_{H^s(\Omega)}+\|u\|_{\dtwo}\right].
\end{align*}
If we further have $u|_{X_n}=0$, then the constant $C_{s,d,\Omega}$ can be chosen to only depend on $\lfloor s\rfloor$, $d,$ and $\Omega$. 
\end{lemma}

As previously noted, for $\lambda >0$ the approximate kriging interpolant $\Interpsi^\nugg f$ does not exactly interpolate the observed function values. The second lemma quantifies the resulting interpolation error for $f \in \RKHSpsi.$

\begin{lemma}[{{\cite[Proposition 3.1]{wendland2005approximate}}}]\label{lemma:regularized LS}
    Let $f\in \RKHSpsi$ with well-defined pointwise values. If $K_\Psi$ is invertible, then
    \begin{align*}
         \| f - \Interpsi^\nugg f \|_{\dtwo} &\leq \sqrt{\nugg}\|f\|_{\RKHSpsi},\\
        \|\Interpsi^\nugg f\|_{\RKHSpsi}& \leq \|f\|_{\RKHSpsi}.
    \end{align*}
\end{lemma}

The third lemma is analogous to Lemma \ref{lemma:regularized LS}, but only assumes $g \in L^\infty(\D)$. This weaker assumption will be leveraged in  Lemma \ref{lemma:stability of interpolation} to establish a stability result.

\begin{lemma}\label{lemma:approx inter error on Xn}
Let $g\in L^\infty(\D)$ with well-defined pointwise values. If $\smin(K_{\Psi})+\nugg>0,$ then
\begin{align*}
    \|g - \Interpsi^{\nugg}g \|_{\dtwo}& \leq \frac{\nugg}{\smin(K_{\Psi})+\nugg} \|g\|_{\dtwo},\\
     \|\Interpsi^\nugg g\|_{\RKHSpsi} &\leq \frac{1}{\sqrt{\smin(K_{\Psi})+\nugg}}\|g\|_{\dtwo}.
\end{align*}
\end{lemma}

Now given the technical lemmas, we are ready to present the building blocks that will lead to the proof of Theorem \ref{thm:misspecifed kriging error}.  
The following lemma provides $L^2(\Omega)$ and $L^\infty(\Omega)$ bounds on kriging interpolants in the well-specified case where $f$ belongs to the RKHS of the kernel.  

\begin{lemma}\label{lemma:approx interpolation error}
Let $f \in \RKHSpsi$ and suppose that $\RKHSpsi$ is norm equivalent to a subspace of $H^s(\D)$ for some $s>\frac{d}{2}$: there exists $A>1$ such that $A^{-1}\|g\|_{H^s(\D)} \leq \|g\|_{\RKHSpsi}\leq A\|g\|_{H^s(\D)}$ for all $g\in \RKHSpsi$. If $K_\Psi$ is invertible, then there exists $h_{\lfloor s\rfloor,d,\Omega}$ such that for $\h\leq h_{\lfloor s\rfloor,d,\Omega}$ we have
    \begin{align*}
    \|f-\Interpsi^\nugg f\|_{L^2(\Omega)}& \leq C_{s,d,\Omega}A^2\big(\h^s+\h^{\frac{d}{2}}\sqrt{\nugg}\big) \|f\|_{H^s(\D)},\\
        \|f-\Interpsi^\nugg f\|_{L^{\infty}(\Omega)}&\leq  C_{s,d,\Omega}A^2\big(\h^{s-\frac{d}{2}}+\sqrt{\nugg}\big) \|f\|_{H^s(\D)}.
    \end{align*}
If $\nugg=0$, then the constant $C_{s,d,\Omega}$ can be chosen to only depend on $\lfloor s\rfloor,d,$ and $\Omega$.
\end{lemma}

\begin{proof}[Proof of Lemma \ref{lemma:approx interpolation error}]
By assumption $f-\Interpsi^\nugg f\in H^s(\D).$ Hence,  Lemma \ref{lemma:sampling theorem} implies that there exists $h_{\lfloor s\rfloor,d,\Omega}$ and $c_{s,d,\Omega}$ such that if $\h\leq h_{\lfloor s\rfloor,d,\Omega}$ we have 
\begin{align*}
    \|f-\Interpsi^\nugg f\|_{L^2(\Omega)} &\leq c_{s,d,\Omega}\left[ \h^s \|f-\Interpsi^\nugg f\|_{H^s(\Omega)} + \h^{\frac{d}{2}}\|f-\Interpsi^\nugg f\|_{\dtwo}\right]\\
    &\leq c_{s,d,\Omega}\left[ \h^s \|f-\Interpsi^\nugg f\|_{H^s(\D)} + \h^{\frac{d}{2}}\sqrt{\nugg}\|f\|_{\RKHSpsi}\right],
\end{align*}
where we have used Lemma \ref{lemma:regularized LS} for the last inequality. By norm equivalence, 
\begin{align*}
    \|f-\Interpsi^\nugg f\|_{H^s(\D)}
    &\leq  A\|f-\Interpsi^\nugg f\|_{\RKHSpsi}\\
    &\leq  A\bigl[\|f\|_{\RKHSpsi}+\|\Interpsi^\nugg f\|_{\RKHSpsi} \bigr]\leq 2A\|f\|_{\RKHSpsi},
\end{align*}
where the last step follows from Lemma \ref{lemma:regularized LS}. Now, by norm equivalence again, 
\begin{align*}
    \|f-\Interpsi^\nugg f\|_{L^2(\Omega)} &\leq C_{s,d,\Omega}A\big(\h^s+\h^{\frac{d}{2}}\sqrt{\nugg}\big) \|f\|_{\RKHSpsi}\\
    &\leq C_{s,d,\Omega}A^2\big(\h^s+\h^{\frac{d}{2}}\sqrt{\nugg}\big) \|f\|_{H^s(\D)}.
\end{align*}
The $L^{\infty}$ error bound is proved similarly. 

If $\nugg=0$, then we know $\Interpsi f$ interpolates $f$ over $X_n$ so that $(f-\Interpsi f)|_{X_n}=0$. Therefore by Lemma \ref{lemma:sampling theorem}, the constant $C_{s,d,\Omega}$ in above can be improved to only depend on $\lfloor s \rfloor$, $d,\Omega$. 
\end{proof}

To address misspecification, we need a result beyond the well-specified setting as in Lemma \ref{lemma:approx interpolation error}. 
The following lemma is the key step towards our part II analysis, which bounds the norm (or the Lebesgue constant) of the interpolation operator $\Interpsi^\lambda$ interpreted as being defined over $L^2(X_n)$.  
The result generalizes \cite{de2010stability} and may be of independent interest. 
\begin{lemma}\label{lemma:stability of interpolation}
Under the same conditions as in Lemma \ref{lemma:approx interpolation error} except that $f$ is only assumed to be in $L^\infty(\D)$ with well-defined pointwise values, the following holds: 
    \begin{align*}
        \|\Interpsi^\lambda f\|_{L^2(\Omega)} &\leq C_{s,d,\Omega} A\bigg[\frac{\h^{s}}{\sqrt{\smin(K_{\Psi})+\lambda}}+\h^{\frac{d}{2}}\bigg] \|f\|_{\dtwo}, \\
         \|\Interpsi^\lambda f\|_{L^\infty(\Omega)} & \leq  C_{s,d,\Omega} A\bigg[\frac{\h^{s-\frac{d}{2}}}{\sqrt{\smin(K_{\Psi})+\lambda}}+1\bigg] \|f\|_{\dtwo}.
    \end{align*}
    If $\nugg=0$, then the constant $C_{s,d,\Omega}$ can be chosen to only depend on $\lfloor s\rfloor,d,$ and $\Omega$.
\end{lemma}
\begin{proof}[Proof of Lemma \ref{lemma:stability of interpolation}]
By Lemma \ref{lemma:sampling theorem} applied to $\Interpsi^\nugg f\in H^s(\D)$, we have 
\begin{align}\label{eq:stability step 1}
    \|\Interpsi^\nugg f\|_{L^2(\Omega)} 
    &\leq C_{s,d,\Omega}\left[\h^s\|\Interpsi^\nugg f\|_{H^s(\Omega)} + \h^{\frac{d}{2}}\|\Interpsi^\nugg f\|_{\dtwo} \right].
\end{align}
Now since $f\in L^\infty(\D)$, norm equivalence and Lemma \ref{lemma:approx inter error on Xn} gives that 
\begin{align*}
    \|\Interpsi^\nugg f\|_{H^s(\Omega)} &\leq A\|\Interpsi^\nugg f\|_{\RKHSpsi}\leq A\frac{\|f\|_{\dtwo}}{\sqrt{\smin(K_{\Psi})+\nugg}},\\
    \|\Interpsi^\nugg f\|_{\dtwo}
    &\leq \|f\|_{\dtwo} +\frac{\nugg}{\smin(K_{\Psi})+\nugg}\|f\|_{\dtwo}\leq 2\|f\|_{\dtwo}.
\end{align*}
Therefore, plugging these estimates in \eqref{eq:stability step 1}, we have 
\begin{align*}
    \|\Interpsi^\nugg f\|_{L^2(\Omega)} \leq C_{s,d,\Omega} A\bigg[\frac{\h^{s}}{\sqrt{\smin(K_{\Psi})+\lambda}}+\h^{\frac{d}{2}}\bigg] \|f\|_{\dtwo}.
\end{align*}   
The $L^\infty(\Omega)$ case can be proved similarly. 
\end{proof}

Now we are ready to prove Theorem \ref{thm:misspecifed kriging error}.

\begin{proof}[Proof of Theorem \ref{thm:misspecifed kriging error}] 
For both parts, since the function $f$ does not necessarily belong to the RKHS of $\Phi_N$, we decompose the interpolation error by considering functions $f_N \in \RKHSN$ for which standard interpolation error bounds hold.

\paragraph{(Part I)}
Let $f_N$ be as given in the statement. 
Since $f_N|_{X_n} \equiv f|_{X_n}$, we have $\Inter^\nugg f = \Inter^\nugg f_N$ and $\InterN^\nugg f = \InterN^\nugg f_N$. 
Therefore, denoting by $\|\cdot\|_{\mathbb{B}}$  either $\|\cdot\|_{\LtwoO}$ or $\|\cdot\|_{\LinfO}$, we can decompose the error as 
\begin{align*}
    \|f  - \InterN^\nugg f \|_{\B} \leq \underbrace{\|f-\Inter^\nugg f\|_{\B}}_{E_1} + \underbrace{\|\Inter^\nugg f_N-f_N\|_{\B}}_{E_2} + \underbrace{\|f_N-\InterN^\nugg f_N\|_{\B}}_{E_3}.
\end{align*}
Importantly, each term $E_i$ is of the form $\|g  - \Interpsi^\nugg g \|_{\B}$ for some kernel $\Psi$ and function $g$ with $g \in \mathcal{H}_\Psi.$ We can hence use Lemma \ref{lemma:approx interpolation error} to bound each term. Let $h_{\lfloor \so\rfloor,\lfloor \sN\rfloor,d,\Omega}$ be the smaller of $h_{\lfloor \so\rfloor,d,\Omega}$ and $h_{\lfloor \sN\rfloor,d,\Omega}$ given in Lemma \ref{lemma:approx interpolation error}. 

Consider the case where $\|\cdot\|_{\B}=\|\cdot\|_{\LtwoO}$. 
Since $f\in \RKHS\subset H^{\so}(\D)$, we have that, for $\h\leq h_{\lfloor \so\rfloor,\lfloor \sN\rfloor,d,\Omega},$ 
\begin{align*}
    E_1 &=\|f-\Inter^\nugg f\|_{\LtwoO}\leq C_{\so,d,\Omega}A_N^2\big(\h^{\so} +\h^{\frac{d}{2}}\sqrt{\nugg}\big) \|f\|_{H^{\so}(\D)}.
\end{align*}
Next, to bound $E_2$ we note that, by assumption, $f_N\in \RKHSN  \subset \RKHS\subset H^{\so}(\D)$. Hence, 
\begin{equation*}%
\begin{aligned}
    E_2 =\|\Inter^\nugg f_N-f_N\|_{\LtwoO}\leq C_{\so,d,\Omega}A_N^2\big(\h^{\so}+\h^{\frac{d}{2}}\sqrt{\nugg}\big) \|f_N\|_{H^{\so}(\D)}.
\end{aligned}
\end{equation*}
Finally, since $f_N\in \RKHSN\subset H^{\sN}(\D)$, 
\begin{equation*}%
\begin{aligned}
    E_3 =\|f_N-\InterN^\nugg f_N\|_{\LtwoO}\leq C_{\sN,d,\Omega}A_N^2\big(\h^{\sN}+\h^{\frac{d}{2}}\sqrt{\nugg}\big)\|f_N\|_{H^{\sN}(\D)}.
\end{aligned}
\end{equation*}
Combining the bounds for $E_1,$ $E_2,$ and $E_3$ gives the desired $\LtwoO$ bound. The $\LinfO$ bound is proved similarly. When $\nugg=0$, the constants $C_{\so,d,\Omega}$ and $C_{\sN,d,\Omega}$ can be replaced by $C_{\lfloor\so\rfloor,d,\Omega}$ and $C_{\lfloor\sN\rfloor,d,\Omega}$ by Lemma \ref{lemma:approx interpolation error}.

\paragraph{(Part II)}
We shall again introduce a function $f_N\in\RKHSN$, but we no longer require it to interpolate $f$. 
Denoting by $\|\cdot\|_{\mathbb{B}}$ either $\|\cdot\|_{\LtwoO}$ or $\|\cdot\|_{\LinfO}$, we have 
\begin{equation*} %
\begin{aligned}
    \| f - \InterN^\nugg f\|_{\mathbb{B}}\leq \underbrace{\|f - f_N\|_{\mathbb{B}}}_{F_1}  
    +   \underbrace{\left\|f_N-\InterN^\nugg f_N\right\|_\mathbb{B}}_{F_2}
   +  \underbrace{\left\| \InterN^\nugg f_N - \InterN^\nugg f \right\|_{\mathbb{B}}}_{F_3}.
\end{aligned}
\end{equation*}
Note that $F_2$ is the well-specified approximate interpolation error for $f_N$, which can be bounded using Lemma \ref{lemma:approx interpolation error}, and $F_3$ is the norm of $\InterN^\nugg (f_N-f)$, which can be bounded using Lemma \ref{lemma:approx inter error on Xn}. 
Let $h_{\lfloor \sN\rfloor,d,\Omega}$ be the constant given in Lemma \ref{lemma:approx interpolation error}. 

Consider first the case where $\|\cdot\|_\mathbb{B}=\|\cdot\|_{\LtwoO}$.
For $F_1$ we simply use that 
\begin{align*}
    F_1=\|f-f_N\|_{\LtwoO} \leq C_\Omega\|f-f_N\|_{\LinfO} \le C_{\Omega}\sqrt{n}\h^{d/2}\|f-f_N\|_{\LinfO},
\end{align*}
where the last inequality holds since $\h\gtrsim n^{-1/d}.$ 

For $F_2$, we note that $\RKHSN \subset H^{\sN} (\D)$, and so Lemma \ref{lemma:approx interpolation error} implies that, for $\h\leq h_{\lfloor \sN\rfloor,d,\Omega},$
\begin{align*}
    F_2=\|f_N-\InterN^\nugg f_N\|_{\LtwoO}\leq C_{\sN,d,\Omega}A_N^2\big(\h^{\sN}+\h^{\frac{d}{2}}\sqrt{\nugg}\big) \|f_N\|_{H^{\sN} (\D)}.
\end{align*}
To bound $F_3$, we apply Lemma \ref{lemma:stability of interpolation} to $\InterN^\nugg g$, where $g:=f_N-f\in L^\infty(\D)$:
\begin{align*}
    F_3= \|\InterN^\nugg g\|_{L^2(\Omega)} 
    &\leq C_{\sN,d,\Omega} A_N\sqrt{n}\bigg[\frac{\h^{\sN}}{\sqrt{\smin(K_{\Phi_N})+\lambda}}+\h^{\frac{d}{2}}\bigg] \|g\|_{\dinf}.
\end{align*} 
Now combining the bounds for $F_1,$ $F_2,$ and $F_3$, we obtain 
\begin{align*}
    \|f-\InterN^\nugg f\|_{L^2(\Omega)}&\leq C_{\sN,d,\Omega}A_N\sqrt{n}\bigg[\frac{\h^{\sN}}{\sqrt{\smin(K_{\Phi_N})+\nugg}}+\h^{\frac{d}{2}}\bigg] \|f_N-f\|_{L^{\infty}(\Omega)}\\  &\qquad\qquad  +C_{\sN,d,\Omega}A_N^2\big(\h^{\sN}+\h^{\frac{d}{2}}\sqrt{\nugg}\big)\|f_N\|_{H^{\sN}(\D)},
\end{align*}
as desired. 
The $\LinfO$ bound is proved similarly. When $\nugg=0$, the constant $C_{\sN,d,\Omega}$ can be improved to $C_{\lfloor\sN\rfloor,d,\Omega}$ by Lemma \ref{lemma:approx interpolation error}.
\end{proof}

\section{Epistemic and computational misspecification} \label{sec:examples}
In this section, we apply our unified framework in Theorem \ref{thm:misspecifed kriging error} to several examples.

\subsection{Epistemic misspecification in the Mat\'ern model} \label{sec:ex:para miss}

\subsubsection{Background}
The Mat\'ern kernel is defined by
\begin{align}\label{eq:Matern}
    \Phi(x,\tilde{x})= \sigma^2\frac{2^{1-\nu}}{\Gamma(\nu)} \left(\kappa|x-\tilde{x}|\right)^{\nu}K_{\nu}\left(\kappa|x-\tilde{x}|\right),\qquad x,\tilde{x}\in\R^d,
\end{align}
where $\Gamma$ is the gamma function and $K_\nu$ is the modified Bessel function of the second kind. 
Here the parameters $\sigma,\nu,$ and $\kappa$ control respectively the marginal variance, smoothness, and correlation lengthscale. Due to its modeling flexibility, the Mat\'ern kernel is widely used in spatial statistics \cite{stein1999interpolation}, inverse problems \cite{roininen2014whittle}, and machine learning \cite{williams2006gaussian}. 
Nevertheless, the performance of kernel methods in these and other applications is sensitive to the choice of parameters, and this choice is often based on potentially erroneous prior beliefs on the function to be modeled.

Let $f\in \RKHS,$ where $\Phi$ is a Matérn kernel with unknown parameters $\sigma_0,\nu_0,$ and $\kappa_0.$
Suppose we find the kriging interpolant of a function $f \in \RKHS$ using a kernel
\begin{align}\label{eq:ex-epistemic-kernel}
    \Phi_N(x,\tilde{x})= \sigma_N^2 \frac{2^{1-\nu_N}}{\Gamma(\nu_N)}(\kappa_N|x-\tilde{x}|)^{\nu_N} K_{\nu_N}(\kappa_N|x-\tilde{x}|),\qquad x,\tilde{x}\in \R^d, 
\end{align}
where $\sigma_N,\nu_N,$ and $\kappa_N$ are estimates of the unknown parameters that may be updated as we acquire more data, in which case $N$ depends on $n.$
We are interested in understanding the effect of such epistemic misspecification on the accuracy of the resulting kriging estimator. 
This question has been studied in previous works, e.g. \cite{teckentrup2020convergence,tuo2020kriging,wynne2021convergence}. We will next show that our unified framework is able to reproduce existing results.

\subsubsection{Error analysis}
We shall apply Theorem \ref{thm:misspecifed kriging error} part I for the error analysis. First, recall that the RKHS associated with the Mat\'ern kernel is norm equivalent to a Sobolev space. Specifically, for the true smoothness parameter $s_0:= \nu_0 + \frac{d}{2}$ and its estimate $s_N:= \nu_N + \frac{d}{2},$ we have
\begin{alignat*}{5}
    &\RKHS \hspace{0.06cm} =H^{\so}(\R^d),\qquad &&A_0^{-1}\|g\|_{H^{\so}(\R^d)} \hspace{0.05cm}\leq \|g\|_{\RKHS} \hspace{0.06cm}\leq \hspace{0.05cm} A_0\|g\|_{H^{\so}(\R^d)},\qquad &&\forall g\in H^{\so}(\R^d)\,,\\
    &\RKHSN \hspace{-0.1cm} =H^{\sN}(\R^d), &&A_N^{-1}\|g\|_{H^{\sN}(\R^d)}\leq \|g\|_{\RKHSN} \hspace{-0.07cm} \leq \hspace{-0.07cm}A_N\|g\|_{H^{\sN}(\R^d)},  &&\forall g\in H^{\sN}(\R^d) ,
\end{alignat*}
for some constants $A_0$ and $A_N$. 
Notice that misspecification of the parameters $\sigma$ and $\kappa$ does not affect the resulting RKHS. Furthermore, if the sequence of parameters $\sigma_N,\nu_N,\kappa_N$ are contained in a compact set $S\subset (0,\infty)^3$, then 
\cite[Lemma 3.4]{teckentrup2020convergence} shows that $A_N$ can be chosen to satisfy
\begin{align}\label{eq:ex:epistemic norm equiv constant}
   \underset{N}{\operatorname{max}}\,\, A_N^2 \leq \underset{N}{\operatorname{max}}\,\,\kappa_N \vee \kappa_N^{-1}<\infty,
\end{align}
so that the norm equivalence constants in Theorem \ref{thm:misspecifed kriging error} are uniformly bounded in $N$. 

The following lemma ensures the existence of a smoother interpolant for Sobolev functions, which corresponds to the function $f_N$ in Theorem \ref{thm:misspecifed kriging error} part I.

\begin{lemma}[{{\cite[Lemma 14]{tuo2020kriging}}}]\label{lemma:ex-epistemic-fN}
Suppose $s_N\geq s_0>\frac{d}{2}$. Then, for each $f^*\in H^{\so}(\mathbb{R}^d)$ there exists $f_N^*\in H^{\sN}(\mathbb{R}^d)$ with $f^*|_{X_n}=f^*_N|_{X_n}$ and 
    \begin{align*}
        \|f^*_N\|_{H^{\sN}(\mathbb{R}^d)} &\leq C\q^{\so-\sN}\|f^*\|_{H^{\so}(\mathbb{R}^d)}, \\
        \|f^*_N\|_{H^{\so}(\mathbb{R}^d)}& \leq C\|f^*\|_{H^{\so}(\mathbb{R}^d)},
    \end{align*}
    where $C$ is a constant independent of $N$. 
\end{lemma}

With this lemma, Theorem \ref{thm:misspecifed kriging error} part I can be applied to yield the following result, which reproduces the $L^2$ and $L^\infty$ bounds from \cite[Theorem 3.5]{teckentrup2020convergence} and \cite[Theorem 4]{wynne2021convergence} in the noiseless case. We present the results in a slightly different way from these two theorems, as we make explicit the convergence order in $N.$ 

\begin{theorem}
Let $f\in H^{\so}(\Omega)$ and let $\Phi_N$  be defined as in \eqref{eq:ex-epistemic-kernel}. 
Suppose that there exist $\underbar{B},\overline{B}>0$ independent of $N$ such that $\underbar{B} <\sigma_N,\nu_N,\kappa_N< \overline{B}$ with $\nu_N$ taking only finitely many values. Then, there exists $h_0$ 
independent of $N$ such that for $\h\leq h_0$, setting $s_N :=\nu_N+\frac{d}{2}$ and choosing $\sqrt{\nugg}\asymp \h^{\operatorname{inf}_N \nu_N}$ we have
\begin{align*}
    \|f-\InterN^\nugg f\|_{L^2(\Omega)} &\leq C \h^{\sN} \q^{(\so-\sN)\wedge 0}\|f\|_{H^{\so}(\Omega)},\\
    \|f-\InterN^\nugg f\|_{L^{\infty}(\Omega)} &\leq C\h^{\sN-d/2} \q^{(\so-\sN)\wedge 0}\|f\|_{H^{\so}(\Omega)},
\end{align*}
where $C$ is a constant independent of $N$ and $n$.  
When $\nugg=0$, the same bounds hold without assuming that $\nu_N$ take only finitely many values. 
\end{theorem}

\begin{proof}
First, Proposition \ref{prop:sobolev ext} implies the existence of $f^*\in H^{\so}(\R^d)$ such that 
\begin{align*}
    f^*|_{\Omega}=f|_{\Omega},\qquad \|f^*\|_{H^{\so}(\R^d)}\leq C\|f^*\|_{H^{\so}(\Omega)},
\end{align*}
where $C$ is independent of $f^*$.
Furthermore, since $\Phi_N$ is positive definite, the kernel matrix $K_{\Phi_N}$ is invertible. 
Next, we consider two cases: $s_N\leq s_0$ and $s_N>s_0$.  

\paragraph{Case 1: $s_N\leq s_0$} In this case $f^*\in H^{\so}(\R^d)=\RKHS\subset \RKHSN=H^{\sN}(\R^d)$, so the error becomes a standard approximate interpolation error. By Lemma \ref{lemma:approx interpolation error}, there exist $h_{\lfloor\sN\rfloor,d,\Omega}$ such that if $\h\leq h_{\lfloor\sN\rfloor,d,\Omega}$ then we have 
 \begin{align*}
    &\|f-\InterN^\nugg f\|_{L^2(\Omega)} = \|f^*-\InterN^\nugg f^*\|_{L^2(\Omega)}\\
    &\leq C_{\sN,d,\Omega}A_N^2\big(\h^{\sN}+\h^{\frac{d}{2}}\sqrt{\nugg}\big) \|f\|_{H^{\sN}(\R^d)}
    \leq C_{\sN,d,\Omega}\big(\h^{\sN}+\h^{\frac{d}{2}}\sqrt{\nugg}\big) \|f\|_{H^{\sN}(\Omega)},
\end{align*}
where we have used \eqref{eq:ex:epistemic norm equiv constant} and similarly 
\begin{align*}
    \|f - \InterN f\|_{\LinfO} 
        &\leq C_{\sN,d,\Omega}\big(\h^{\sN-d/2}+\sqrt{\nugg}\big) \|f^*\|_{H^{\sN}(\Omega)}.
\end{align*}

\paragraph{Case 2: $s_N>s_0$} In this case $f_N^*\in H^{\sN}(\R^d)=\RKHSN\subset \RKHS =H^{\so}(\R^d)$ and we can apply Theorem \ref{thm:misspecifed kriging error} part I to $f^*$ and the $f_N^*$ constructed in Lemma \ref{lemma:ex-epistemic-fN} to obtain that for $\h \leq h_{\lfloor \so\rfloor,\lfloor \sN\rfloor,d,\Omega}$
\begin{equation*}
    \begin{aligned}
         &\|f - \InterN^\nugg f\|_{\LtwoO} 
        =  \|f^* - \InterN^\nugg f^*\|_{\LtwoO} \\
        &\leq C_{\so,\sN,d,\Omega} \Bigl[\big(\h^{\so}+\h^{\frac{d}{2}}\sqrt{\nugg}\big) \bigl(\|f^*\|_{H^{\so}(\R^d)}+\|f_N^*\|_{H^{\so}(\R^d)}\bigr) \\
        & \hspace{6.65cm} + \big(\h^{\sN}+\h^{\frac{d}{2}}\sqrt{\nugg}\big)\|f_N^*\|_{H^{\sN}(\R^d)}\Bigr]\\
        &\leq C_{\so,\sN,d,\Omega}\left[ \big(\h^{\so}+\h^{\frac{d}{2}}\sqrt{\nugg}\big) \|f^*\|_{H^{\so}(\R^d)}+\big(\h^{\sN} +\h^{\frac{d}{2}}\sqrt{\nugg}\big)\q^{\so-\sN} \|f^*\|_{H^{\so}(\R^d)}\right]\\
        &\leq C_{\so,\sN,d,\Omega}\big(\h^{\sN}+\h^{\frac{d}{2}}\sqrt{\nugg}\big) \q^{\so-\sN} \|f^*\|_{H^{\so}(\Omega)},
    \end{aligned}
\end{equation*}
where we have used that $\q\leq \h$ in the last step, and similarly
\begin{equation*}
    \begin{aligned}
        \|f - \InterN f\|_{\LinfO} 
        &\leq C_{\so,\sN,d,\Omega}\big(\h^{\sN-d/2}+\sqrt{\nugg}\big) \q^{\so-\sN} \|f^*\|_{H^{\so}(\Omega)}.
    \end{aligned}
\end{equation*}
Since by assumption $\nu_N$ takes only finitely many values, the constants $C_{\so,\sN,d,\Omega}$ for different $N$ can be uniformly upper bounded by a constant $C$ independent of $N$ and $n$. Furthermore, the numbers $h_{\lfloor \so\rfloor,\lfloor \sN\rfloor,d,\Omega}$ can be lower bounded by a number $h_0$ independent of $N$. 
The first assertions follow by combining both cases and setting $\sqrt{\nugg}\asymp \h^{\operatorname{inf}_N \nu_N}$.

If $\nugg=0$, then Theorem \ref{thm:misspecifed kriging error} suggests that the constants $C_{\so,\sN,d,\Omega}$ can be improved to $C_{\lfloor\so\rfloor,\lfloor\sN\rfloor,d,\Omega}$. 
Since now the new constants only depend on the integer parts of the $s_N$'s, which are uniformly bounded above and below, they can be uniformly upper bounded by a constant $C$ independent of $N$ and $n$. Similarly, a uniform lower bound $h_0$ independent of $N$ for the $h_{\lfloor \so\rfloor,\lfloor \sN\rfloor,d,\Omega}$'s can be obtained. 
\end{proof}

\subsection{Karhunen-Lo\`eve expansions}\label{sec:KL}

\subsubsection{Background}\label{ssec:KLbackground}
Let $u$ be the Gaussian field formally defined via the SPDE
\begin{align}\label{eq:SPDE bounded domain}
    (\kappa^2 - \Delta)^{s/2} u = \kappa^{s-d/2}\mathcal{W},\qquad x\in \D, 
\end{align}
supplemented with suitable boundary conditions,
where the fractional elliptic operator is defined spectrally and $\mathcal{W}$ is a spatial white noise with unit variance. The parameters $\kappa$ and $s$ in \eqref{eq:SPDE bounded domain} control the inverse lengthscale and smoothness of the field $u.$ We refer to \cite{lindgren2011explicit,sanz2022finite} for further details on the definition of Matérn-type Gaussian fields via SPDEs. For our purposes, it suffices to note that the covariance function of the field $u$ defined formally via \eqref{eq:SPDE bounded domain} is given by
\begin{align*}
    \Phi(x,\tilde{x}) = \sum_{i=1}^{\infty}(\kappa^2+\lambda_i)^{-s}\psi_i(x)\psi_i(\tilde{x}), \qquad x,\tilde{x}\in \D,
\end{align*}
where $(\lambda_i,\psi_i)$'s are the ordered eigenpairs of the associated Laplacian $-\Delta$ endowed with suitable boundary conditions. 
For computation, one may use the kernel
\begin{align}\label{eq:ex-KL trun kernel}
    \Phi_N(x,\tilde{x}) = \sum_{i=1}^{N}(\kappa^2+\lambda_i)^{-s}\psi_i(x)\psi_i(\tilde{x}), \qquad x,\tilde{x}\in\D, 
\end{align}
where $N \ge 1$ is a truncation parameter.
Such Karhunen-Lo\`eve expansions are appealing when explicit formulas are available for the Laplacian eigenpairs (e.g. in hyperrectangles or in the sphere \cite{sanz2022finite,lang2015isotropic}). Moreover, fast algorithms for elliptic eigenvalue problems have broadened their scope, making them attractive in large data applications  \cite{greengard2022equispaced}. 
We remark that \eqref{eq:ex-KL trun kernel} is also an instance of the Mercer kernel \cite{opfer2006multiscale}. 
The RKHSs associated with $\Phi$ and $\Phi_N$ are given by
\begin{align*}
    \RKHS&=\left\{ g = \sum_{i=1}^{\infty}a_i\psi_i: \|g\|_{\RKHS} :=  \sum_{i=1}^{\infty}a_i^2(\kappa^2+\lambda_i)^s<\infty\right\},\\
    \RKHSN &=\left\{ g = \sum_{i=1}^N a_i\psi_i:  \|g\|_{\RKHSN } :=\sum_{i=1}^Na_i^2(\kappa^2+\lambda_i)^s<\infty\right\}.
\end{align*}
In practice, the domain $\D$ is chosen so that the experimental region $\Omega$ is compactly contained in $\D.$ By doing so, one avoids boundary artifacts near the boundary of $\Omega,$ thus obtaining kriging interpolants that are accurate up to the boundary  \cite{khristenko2019analysis,harlim2022graph}.

\subsubsection{Error analysis: general domain}\label{ssec:KLgeneraldomain}

In this subsection, we consider a general domain $\D$ and define the covariance function by equipping \eqref{eq:SPDE bounded domain} with Neumann boundary condition. We focus on computational misspecification and assume that the smoothness parameter $s$ in the definition of $\Phi_N$ matches the true smoothness $s_0$.

Following \cite{dunlop2020large}, we define for $m\in\mathbb{N}$ the function spaces 
\begin{align*}
    H^{2m}_{\NB}(\D)=\left\{u\in H^{2m}(\D): \frac{\partial (-\Delta)^r u}{\partial n}=0 \text{ for all } 0 \leq r \leq m-1 \text{ on } \partial \D\right\}
\end{align*}
and $H^{2m+1}_{\NB}(\D)=H^{2m+1}(\D)\cap H^{2m}_{\NB}(\D)$. A similar argument as in \cite[Lemma 17]{dunlop2020large} (adapting to the additional $\kappa$ term) shows that  $\RKHS=H^s_{\NB}(\D)$ with equivalent norms. Moreover,
$\|\cdot\|_{\RKHSN}$ agrees with $\|\cdot\|_{\RKHS}$ on $\RKHSN,$ so that
\begin{align}\label{eq:norm equiv KL general}
    A^{-1}\|g\|_{\RKHSN}=A^{-1}\|g\|_{\RKHS} \leq \|g\|_{H^s(\D)} \leq A\|g\|_{\RKHS}=A \|g\|_{\RKHSN},\quad \forall g\in \RKHSN, 
\end{align}
for some constant $A$ independent of $N$. By \cite[Theorem 2]{schaback2002approximation}, the kernel $\Phi$ is positive definite  since the space $H_{\NB}^s(\D)$ separates points.  

To carry out the analysis, given a function $f$ defined over $\Omega$, we will extend it first to be in $\RKHS$ and then apply Theorem \ref{thm:misspecifed kriging error} part I. The following lemma establishes the necessary technical steps.

\begin{lemma}\label{lemma:ex-kl-general-fN}
   Let $f\in H^s(\Omega)$ with $s>\frac{d}{2}$ an integer. There exists a function $f^*\in H^s_{\NB}(\D)=\RKHS$ such that
        \begin{align}            f^*|_{\Omega}=f|_{\Omega},\qquad \|f^*\|_{H^s(\D)}\leq C \|f\|_{H^s(\Omega)}, 
        \end{align}    
     where $C$ is a constant independent of $f$. 
     
     Suppose $\operatorname{sup}_{i\in\mathbb{N}}\|\psi_i\|_{\infty} \leq B$ for some constant $B$. Let $N$ be large enough so that
     \begin{align}\label{eq:KL general N}
    \sum_{j=N+1}^{\infty}(\kappa^2+\lambda_j)^{-s}\leq \frac{1}{\gamma^2} \frac{\sigma_{\min}(K_\Phi)}{nB^2},
\end{align}
where $\gamma>1$ is a fixed constant and $\sigma_{\min}(K_\Phi)$ denotes the smallest eigenvalue of $K_\Phi.$ 
 Then, there exists a function $f^*_N\in \RKHSN$ such that 
        \begin{align}
            f^*_N|_{X_n}=f^*|_{X_n},\qquad \|f^*_N\|_{H^s(\D)}\leq C \|f\|_{H^s(\Omega)}, 
        \end{align}
        where $C$ is a constant independent of $f,N,$ and $n$. 
\end{lemma}
\begin{proof}
To prove the first claim, we first extend $f$ to $\R^d$ using Proposition \ref{prop:sobolev ext} and then control its support with a bump function. 
More precisely, Proposition \ref{prop:sobolev ext} implies the existence of a function $F\in H^s(\R^d)$ such that
\begin{align}\label{eq:KL general eqn1}
    F|_{\Omega}=f|_{\Omega},\qquad \|F\|_{H^s(\R^d)}\leq C \|f\|_{H^s(\Omega)}
\end{align}
for a constant $C$ independent of $f$. 
Let $\Omega^*$ be an open set satisfying $\Omega \ssubset \Omega^* \ssubset \D$ and let $\eta\in C^{\infty}(\R^d)$ satisfy $\eta=1$ on $\Omega$ and $\eta=0$ on $(\Omega^*)^c$. 
Then, the function  $f^*:=(\eta F)|_{\D}$ belongs to $ H^s(\D)$ and is compactly supported in $\D$, which implies that $f^*\in H^s_{\NB}(\D)$.
We also have  $f^*|_{\Omega}=F|_{\Omega}$, which together with \eqref{eq:KL general eqn1} implies that $f^*|_{\Omega}=f|_{\Omega}$. Moreover, Leibniz rule of differentiation gives that 
\begin{align*}
    \|f^*\|_{H^s(\D)}\leq C\|F\|_{H^s(\D)}&\leq C \|F\|_{H^s(\R^d)}\leq C\|f\|_{H^s(\Omega)}, 
\end{align*}
where $C$ depends on the $L^{\infty}(\D)$ norm of the derivatives of $\eta$ but is independent of $f$.

For the second claim, note that $f^*\in H_{\NB}^s(\D)=\RKHS,$ and so by Proposition \ref{prop:exist interpolant mercer kernel} there exists $f_N^*\in \RKHSN$ such that 
\begin{align*}
    f^*_N|_{X_n}=f^*|_{X_n},\qquad \|f^*_N\|_{\RKHSN}=\|f^*_N\|_{\RKHS}\leq C\|f^*\|_{\RKHS},
\end{align*}
which by norm equivalence \eqref{eq:norm equiv KL general} further implies that 
\begin{align*}
    \|f_N^*\|_{H^s(\D)} \leq C \|f_N^*\|_{\RKHSN} 
    \leq C\|f^*\|_{\RKHS}\leq C\|f^*\|_{H^s(\D)}\leq  C\|f\|_{H^s(\Omega)}. 
\end{align*}
\end{proof}

Before presenting the main result of this subsection, we remark that the existence of $f^*_N\in \RKHSN$ that interpolates $f^*$ in Lemma \ref{lemma:ex-kl-general-fN} implies invertibility of the kernel matrix $K_{\Phi_N}$. Indeed, let $Y\in\R^n$ be a vector and let $f^*\in \RKHS$ satisfy $f^*(x_i)=Y_i$ for $i=1,\ldots,n$ (simply take $f^*$ as a $\Phi$ interpolant of $Y$). Existence of $f_N^*\in \RKHSN$ satisfying $f_N^*|_{X_n}=f^*|_{X_n}$ implies that the set 
\begin{align*}
    \mathcal{H}_0:=\{g\in \RKHSN:g|_{X_n}=Y\}
\end{align*}
is nonempty, so that the optimization problem 
\begin{align*}
    \underset{g\in \RKHSN}{\operatorname{min}}\,\, \|g\|_{\RKHSN},\qquad \text{subject to } g|_{X_n}=Y
\end{align*}
has a solution $g^*$ that can be written as $g^*=\sum_{i=1}^n \alpha_i \Phi_N(\cdot,x_i)$ for some $\alpha\in\R^n$ (see e.g. \cite[proof of Theorem 3.5]{kanagawa2018gaussian}). In particular, $g^*|_{X_n}=Y$, which written in matrix form is $K_{\Phi_N}\alpha=Y$. Since such an $\alpha$ exists for any $Y\in\R^n$, the invertibility of $K_{\Phi_N}$ follows. Now we are ready to present the main result of this subsection. 
\begin{theorem}\label{thm:KL general}
    Let $f\in H^{s_0}(\Omega)$ with $s_0> \frac{d}{2}$ an integer and let $s=s_0$ in the definition of $\Phi_N$ \eqref{eq:ex-KL trun kernel}. Suppose that $\operatorname{sup}_{i\in\mathbb{N}}\|\psi_i\|_{\infty} \leq B$  and that $N$ satisfies \eqref{eq:KL general N}. Then, there exists $h_0$ independent of $N$ and $n$ such that if $\h\leq h_0$, we have by setting $\sqrt{\nugg}\asymp \h^{\so-d/2}$ that
    \begin{align*}
    \| f - \InterN^\nugg f \|_{\LtwoO}&\leq C \h^{\so}\|f\|_{H^{\so}(\Omega)},\\
        \| f - \InterN^\nugg f \|_{\LinfO}&\leq C\h^{\so-\frac{d}{2}}\|f\|_{H^{\so}(\Omega)}, 
    \end{align*}
    where $C$ is independent of $N$ and $n$.
In particular, if the design points satisfy $\h\lesssim n^{-1/d}$ (such as i.i.d. sampling from some distribution over $\Omega$), then 
\begin{align*}
    \| f - \InterN^\nugg f \|_{\LtwoO}&\leq  Cn^{-\frac{\so}{d}} \|f\|_{H^{\so}(\Omega)},\\
        \| f - \InterN^\nugg f \|_{\LinfO}&\leq  Cn^{-(\frac{\so}{d}-\frac12)} \|f\|_{H^{\so}(\Omega)},
\end{align*}
which are the best possible rates for approximating functions in $H^{\so}(\Omega)$. 
\end{theorem}

\begin{proof}
Let $f^*\in H^{\so}_{\NB}(\D)=\RKHS$ and $f^*_N\in\RKHSN$ be the functions given by Lemma \ref{lemma:ex-kl-general-fN}. The results then follow from Theorem \ref{thm:misspecifed kriging error} part I with $\sqrt{\nugg}\asymp \h^{\so-\frac{d}{2}}$.
\end{proof}

\begin{remark}
We remark that \eqref{eq:KL general N} is an essential condition that guarantees an interpolant $f_N$ exists in Theorem \ref{thm:misspecifed kriging error} part I and similar requirements have been imposed in \cite{griebel2015multiscale}. 
However, \eqref{eq:KL general N} is not explicit unless one can obtain a lower bound on $\smin(K_\Phi)$. 
Since $\Phi$ is of Mat\'ern type, we conjecture that $\smin(K_\Phi)\gtrsim \q^{2s-d}$ (see e.g. \cite[Section 12.2]{wendland2004scattered}) so that we need $N\gtrsim n^{\frac{d}{2s-d}}\q^{-d}$, which for quasi-uniform designs reduces to $N\gtrsim n^{\frac{2s}{2s-d}}$. 
On the other hand, Theorem \ref{thm:misspecifed kriging error} part II analysis could yield error bounds explicit in $N$ which can be tuned based on $n$. We shall demonstrate this idea on a one-dimensional example in the next subsection, leading to Theorem \ref{thm:KL trig}, and also for the examples in Subsections \ref{sec:wavelets} and \ref{sec:FEM}.
$\hfill \square$
\end{remark}

\subsubsection{Error analysis: trigonometric polynomials}\label{ssec:KLtrigonometric}
In the previous subsection, we illustrated an application of Theorem \ref{thm:misspecifed kriging error} part I for pure computational misspecification. In this subsection, we apply Theorem \ref{thm:misspecifed kriging error} part II to simultaneously study computational and epistemic misspecification. 

Let $\Omega=(a,b)\ssubset  (0,1)=\D$ and consider the periodic boundary condition in \eqref{eq:SPDE bounded domain}. Then, the induced covariance function takes the form 
\begin{align*}
    \Phi_N(x,\tilde{x})=1 &+\sum_{k=1}^{N} (1+4\pi^2k^2)^{-s} \cos(2\pi k x) \cos(2\pi k \tilde{x})\\
    &+\sum_{k=1}^{N} (1+4\pi^2k^2)^{-s} \sin(2\pi k x) \sin(2\pi k \tilde{x}),\qquad x,\tilde{x} \in \D
\end{align*}
for $s>\frac12$ an integer.
The RKHS associated with $\Phi_N$ is given by
\begin{align*}
    \RKHSN  &= \Big\{g=a_0+\sum_{k=1}^N a_k\cos(2\pi k\cdot)+b_k\sin(2\pi k\cdot):\\
    &\qquad \qquad \|g\|^2_{\RKHSN }:=a_0^2+\sum_{k=1}^N (a_k^2+b_k^2)(1+4\pi^2k^2)^s<\infty\Big\}\\
    &=\operatorname{span} \Big\{1,\cos(2\pi x), \sin(2\pi x),\ldots,\cos(2\pi Nx),\sin(2\pi Nx)\Big\},
\end{align*}
which is the space of trigonometric polynomials with degree less than or equal to $N$. We claim that, for $s \in \mathbb{N},$ the norm $\| \cdot \|_{\RKHSN }$ is equivalent to the norm $\| \cdot \|_{H^s(\D)}$. Indeed, for $g\in \RKHSN $ and $m \in \mathbb{N},$ 
\begin{align*}
    \|g^{(m)}\|_2^2 = \sum_{k=1}^N (4\pi^2 k^2)^m (a_k^2+b_k^2), 
\end{align*}
so that 
\begin{align*}
    \|g\|^2_{H^s(\D)}=\sum_{m=0}^s \|g^{(m)}\|_2^2 =a_0^2 + \sum_{k=1}^N(a_k^2+b_k^2)\sum_{m=0}^s (4\pi^2 k^2)^m.  
\end{align*}
Since $\sum_{m=0}^s (4\pi^2 k^2)^m \leq (1+4\pi^2k^2)^s \leq \binom{s}{\lfloor s/2 \rfloor} \sum_{m=0}^s (4\pi^2 k^2)^m$, we have 
\begin{align}\label{eq:EX-KL norm equiv}
    \|g\|_{H^s(\D)} \leq \|g\|_{\RKHSN } \leq C_s \|g\|_{H^s(\D)},
\end{align}
where $C_s$ is a universal constant depending only on $s$. 

The following lemma gives a simple condition for invertibility of the kernel matrix $K_{\Phi_N}$, which is needed in Theorem \ref{thm:misspecifed kriging error}.
\begin{lemma}\label{lemma:ex-KL1d-PD}
    Let $X_n=\{x_1,\ldots,x_n\}$ be distinct points. If $n\leq 2N+1$, then the kernel matrix $K_{\Phi_N}$ is invertible. 
\end{lemma}
\begin{proof}
    Classical trigonometric interpolation results (see e.g. \cite{zygmund2002trigonometric}) imply that if $n\leq 2N+1$, then there exists $g_i\in \RKHSN$ such that $g_i(x_j)=\delta_{ij}$ for each $i$. The invertibility of the kernel matrix follows from Lemma \ref{lemma:positive definiteness}. 
\end{proof}

To apply Theorem part II, we will use the following approximation theory result characterizing how well a given function can be approximated by functions in $\RKHSN.$
\begin{lemma}\label{lemma:ex-KL1d-fN and oneN}
    Suppose $f^*\in C^{s_0}(\overline{D})$ with $s_0$ an integer and $f^*(0)=f^*(1)$. There exists $f_N^*\in\RKHSN$ such that 
    \begin{align*}
        \|f^*-f^*_N\|_{L^{\infty}(\D)} &\leq C (\log N)N^{-s_0} \|f^*\|_{C^{\so}(\overline{D})},\\
        \|f_N^*\|_{H^s(\D)}&\leq CN^{(s-s_0)\vee 0} \|f^*\|_{C^{\so}(\overline{D})},
    \end{align*}
    where $C$ is a constant independent of $N$. 
\end{lemma}
\begin{proof}
    Since by assumption $f^*\in C^{s_0}(\overline{\D})$ and is periodic, \cite[Theorem 4.1.2, Lemma 4.1.4(i)]{dai2013approximation} implies that 
    \begin{align*}
        \underset{T\in\RKHSN }{\operatorname{inf}}\,\, \|T-f^*\|_{L^{\infty}(\D)} \leq 
        C N^{-s_0} \|(f^*)^{(s_0)}\|_{L^{\infty}(\D)}
        \leq CN^{-s_0}\|f^*\|_{C^{s_0}(\overline{\D})},
    \end{align*}
    where $C$ is independent of $N$. 
    Take $f^*_N$ to be the $N$-th partial sum of the Fourier series of $f^*$, i.e., 
    \begin{align*}
        f^*_N=a_0 +\sum_{k=1}^N a_k\cos(2\pi k\cdot)+b_k\sin(2\pi k\cdot),\quad a_k=\langle f^*,\cos(2\pi k\cdot)\rangle, b_k=\langle f^*,\sin(2\pi k\cdot)\rangle.
    \end{align*}
    Then, 
    \cite[Theorem 2.2]{rivlin1981introduction} implies that 
    \begin{equation*}
    \begin{aligned}
        \|f^*-f^*_N\|_{L^{\infty}(\D)} \leq C (\log N) \underset{T\in\RKHSN }{\operatorname{inf}}\,\, \|T-f^*\|_{L^{\infty}(\D)} 
        \leq C (\log N)N^{-s_0} \|f^*\|_{C^{\so}(\overline{\D})}.
    \end{aligned}
    \end{equation*}
    Moreover, we have by norm equivalence \eqref{eq:EX-KL norm equiv} that 
    \begin{align*}
        \|f_N^*\|^2_{H^s(\D)}\leq C\|f_N^*\|^2_{\RKHSN} &= C\Big[ a_0^2 + \sum_{k=1}^N (a_k^2+b_k^2)(1+4\pi^2k^2)^s\Big]\\
        &= C\Big[a_0^2 + \sum_{k=1}^N (a_k^2+b_k^2)(1+4\pi^2k^2)^{s_0}  (1+4\pi^2k^2)^{s-s_0}\Big]\\
        &\leq CN^{(2s-2s_0)\vee 0} \Big[a_0^2 + \sum_{k=1}^{\infty} (a_k^2+b_k^2)(1+4\pi^2k^2)^{s_0} \Big]\\
        &\leq CN^{(2s-2s_0)\vee 0} \|f^*\|^2_{H^{\so}(\D)}\leq CN^{(2s-2s_0)\vee 0} \|f^*\|^2_{C^{\so}(\overline{\D})}.
    \end{align*}
\end{proof}

Now we are ready to state the main result of this subsection. 
\begin{theorem}\label{thm:KL trig}
    Suppose $\Omega=(a,b)\ssubset (0,1)=\D$ and $f\in C^{\so}(\overline{\Omega})$ with $s_0$ an integer. 
    Suppose $K_{\Phi_N}$ is invertible. Then, there exists $h_0$ independent of $N$ and $n$ such that if $\h\leq h_0$, we have by setting  $\sqrt{\lambda}\asymp \h^{s-1/2}$ that 
    \begin{align*}
        \|f - \InterN^\nugg f\|_{L^2(\Omega)} &\leq  C\big[\sqrt{n}\h^{\frac12}(\log N)N^{-\so}+\h^sN^{(s-\so)\vee 0}\big]\|f\|_{C^{\so}(\overline{\Omega})},\\
        \|f - \InterN^\nugg f\|_{L^{\infty}(\Omega)} &\leq  C\big[\sqrt{n}(\log N)N^{-\so}+\h^{s-\frac12} N^{(s-\so)\vee 0}\big]\|f\|_{C^{\so}(\overline{\Omega})},
    \end{align*}
    where $C$ is independent of $N$ and $n$.
    If the design points satisfy $\h\lesssim n^{-1}$, setting $2N+1=n$ so that invertibility is ensured as in Lemma \ref{lemma:ex-KL1d-PD}, we have 
    \begin{align*}
        \|f - \InterN^\nugg f\|_{L^2(\Omega)} & \leq C (\log n) n^{-(s_0\wedge s)}\|f\|_{C^{\so}(\overline{\Omega})},\\
         \|f - \InterN^\nugg f\|_{L^{\infty}(\Omega)} & \leq C (\log n) n^{-(s_0\wedge s-\frac12)} \|f\|_{C^{\so}(\overline{\Omega})}.
    \end{align*}
\end{theorem}

\begin{proof}
We first extend $f$ to be a periodic function $f^*$ defined over $\D$. 
By \cite[Theorem 4]{stein1970singular}, there exists $F\in C^{s_0}(\R)$ such that 
\begin{align*}
    F|_{\Omega}=f|_{\Omega},\qquad \|F\|_{C^{\so}(\R)} \leq C \|f\|_{C^{\so}(\overline{\Omega})},
\end{align*}
where $C$ is independent of $f$. 
Next, let $I$ be an interval satisfying $\Omega \ssubset  I \ssubset  \D$ and $\eta\in C^{\infty}(\R)$ satisfy $\eta=1$ on $\Omega$ and $\eta=0$ on $I^c$. 
The function $f^*=(\eta F)|_{\overline{\D}}$ belongs to $C^{s_0}(\overline{\D})$ and satisfies
\begin{align*}
    f^*|_{\Omega}=F|_{\Omega}=f|_{\Omega},\quad \, \,\, \, f^*(0)=f^*(1)=0, \quad \, \, \, \, 
    \|f^*\|_{C^{\so}(\overline{\D})}\leq C\|F\|_{C^{\so}(\R)}\leq C \|f\|_{C^{\so}(\overline{\Omega})}.
\end{align*}
The results follow from Theorem \ref{thm:misspecifed kriging error} part II applied to $f^*$ and the function $f_N^*\in\RKHSN$ from Lemma \ref{lemma:ex-KL1d-fN and oneN}. 
\end{proof}

\subsection{Wavelet-based multiscale kernels}\label{sec:wavelets}

\subsubsection{Background}
In this subsection, we consider a wavelet-based multiscale kernel that can introduce sparsity structures into the kernel matrix for computational speedup. Formally, we consider kernels of the form \begin{align}\label{eq:wavelet kernel}
    \Phi_N(x,\tilde{x}) :=\sum_{k\in\Z^d}\phi_k(x)\phi_k(\tilde{x})+\sum_{j=0}^{N} 2^{-2js} \sum_{\iota\in\I}\sum_{k\in \mathbb{Z}^d} \psi_{jk}^\iota(x)\psi_{jk}^\iota(\tilde{x}), \quad x,\tilde{x}\in \R^d,
\end{align}
where $\{\phi_k,\psi_{jk}^\iota\}_{j\in\mathbb{N},k\in \Z^d,\iota\in \I}$ form an orthonormal basis of $L^2(\R^d)$ and $N \ge 1$ is a truncation parameter. Importantly, the $\phi_k$'s and $\psi_{jk}^\iota$'s can be chosen to have compact support whose size shrinks as $j$ increases. Consequently, points $x,\tilde{x}$ that are far apart would give a zero correlation $\Phi_N(x,\tilde{x})=0$, leading to a sparse covariance matrix. We remark that \eqref{eq:wavelet kernel} is an analog of the multiscale kernel proposed by \cite{opfer2006multiscale,griebel2015multiscale} which instead only uses the scaling function; see also their extensions to general metric measure spaces \cite{griebel2018regularized} and other wavelet-based kernels \cite{bolin2013comparison}. In the following, we analyze the approximate kriging error of using \eqref{eq:wavelet kernel} as the kernel with a different proof strategy from \cite{griebel2015multiscale}.

To describe the class of kernels \eqref{eq:wavelet kernel} in more detail, we now provide the necessary background on wavelets. Let $\phi$ and $\{\psi^{\iota}\}_{\iota\in \I}$, where $\I$ an index set of cardinality $2^d-1$, be the scaling and wavelet functions that satisfy
\begin{enumerate}
    \item[(W1)] $\phi$ and $\{\psi^\iota\}_{\iota\in \I}$ are compactly supported; 
    \item[(W2)] $\phi$ and $\{\psi^\iota\}_{\iota\in \I}$ belong to $\C^r(\mathbb{R}^d)$ for some integer $r$; and
    \item[(W3)] There exists $\delta>0$ and $\frac{d}{2}<s<r$ such that $\int_{|\xi|<\delta} |\widehat{\psi}^\iota(\xi)|^2|\xi|^{-2s}d\xi<\infty$ for all $\iota\in\I$.
    
\end{enumerate}
For instance, the Daubechies wavelets \cite[Section 4.2.3]{gine2021mathematical} satisfy these requirements. 
We recall that the scaling and wavelet functions form a multiresolution analysis of $L^2(\mathbb{R}^d)$: there exist spaces $\{V_i\}_{i=0}^{\infty}$ such that 
\begin{align*}
    V_0\subset V_1 \subset \cdots \subset L^2(\mathbb{R}^d)
\end{align*}
with $\bigl\{\phi_{jk}=2^{jd/2}\phi(2^j\cdot-k)\bigr\}_{k\in\mathbb{Z}^d}$ forming an orthonormal basis for $V_j$ and $\{\psi_{jk}^\iota=2^{jd/2}\psi^\iota(2^j\cdot-k)\}_{k\in\mathbb{Z}^d,\iota\in\I}$ an orthonormal basis for $V_{j+1}\ominus V_j$ (the orthogonal complement of $V_j$ in $V_{j+1}$).

Notice that the $\psi_{jk}^\iota$'s are dilation and translations of the function $\psi^\iota$ whose support size decreases exponentially as $j$ increases. Hence, as $j$ grows these functions represent finer and finer structures of the function space, which motivates referring to \eqref{eq:wavelet kernel} as a multiscale kernel.  
For $s < r$, \cite[p.48 Theorem 8]{meyer1992wavelets} implies (by $r-$regularity ensured by (W1) and (W2)) that, for some constant $A>1,$ 
\begin{align}\label{eq:wavelet sobolev norm}
    A^{-1}\|f\|_{H^s(\mathbb{R}^d)}^2 \leq \sum_{k\in \Z^d}|\langle f,\phi_k\rangle|^2 +\sum_{j\geq 0}\sum_{\iota\in \I}\sum_{k\in\mathbb{Z}^d} |\langle f,\psi_{jk}^\iota\rangle|^22^{2js} \leq A  \|f\|_{H^s(\mathbb{R}^d)}^2 .
\end{align}
Importantly, \eqref{eq:wavelet sobolev norm} allows us to establish norm equivalence between a Sobolev space and the RKHS of $\Phi_N$, given by
\begin{align*}
    \RKHSN  = \bigg\{g=\sum_{k\in\Z^d}a_k\phi_k+\sum_{j=0}^N&\sum_{\iota\in\I}\sum_{k\in\mathbb{Z}^d} b_{jk}^\iota\psi_{jk}\iota: \\&\|g\|^2_{\RKHSN }:=\sum_{k\in \Z^d}a_k^2+\sum_{j=0}^N2^{2js}\sum_{\iota\in \I}\sum_{k\in\mathbb{Z}^d}|b_{jk}^\iota|^2<\infty \bigg\}.
\end{align*}
It then follows from \eqref{eq:wavelet sobolev norm} that, for all $g\in \RKHSN,$
\begin{align}\label{eq:nex-wavvelet-norm equivalence}
    A^{-1}\|g\|_{H^s(\R^d)}\leq \|g\|_{\RKHSN} \leq A \|g\|_{H^s(\R^d)},
\end{align}
and so $\RKHSN\subset H^s(\R^d).$ Notice that here $A$ is independent of $N$.

\subsubsection{Error analysis}

We shall apply Theorem \ref{thm:misspecifed kriging error} part II for the error analysis. As in the previous subsection, we will need to find a good approximation in $\RKHSN$ of a given function. This is covered by the following lemma.  
\begin{lemma}\label{lemma:ex-wavelet:fn onen}
    Let $f^*\in H^{\so}(\R^d)$. There exists $f_N^*\in \RKHSN$ such that,  for all $s>\frac{d}{2},$
    \begin{align*}
        \|f^*-f_N^*\|_{L^{\infty}(\R^d)}& \leq C2^{-N(s\wedge\so-\frac{d}{2})}\|f^*\|_{H^{\so}(\R^d)},\\
        \|f_N^*\|_{H^{s}(\R^d)}&\leq C 2^{N(s-\so)\vee 0}\|f^*\|_{H^{\so}(\R^d)},
    \end{align*}
    where $C$ is a constant independent of $N$. 
\end{lemma}
\begin{proof}
Consider the truncated series expansion 
\begin{align*}
    f_N^*:=\sum_{k\in Z^d}\langle f^*,\phi_k\rangle\phi_k+\sum_{j=0}^N\sum_{\iota\in \I}\sum_{k\in\mathbb{Z}^d}\langle f^*,\psi_{jk}^\iota\rangle \psi_{jk}^\iota.
\end{align*}
To bound the approximation error, we shall apply \cite[Theorem 1]{kon2001convergence} and check the conditions. 
Notice that (W1) ensures that $\phi$ and $\{\psi^\iota\}_{\iota\in\I}$ are all radially bounded and (W3) holds for $s$ replaced by $s\wedge s_0$. So we have 
\begin{equation}\label{eq:ex-wavelet-fN}
\begin{aligned}
    \|f^*-f^*_N\|_{L^{\infty}(\R^d)} \leq C2^{-N(s\wedge\so-\frac{d}{2})}\|f^*\|_{H^{\so}(\mathbb{R}^d)},
\end{aligned}
\end{equation}
where $C$ is a constant independent of $N$. 
Moreover, by norm equivalence \eqref{eq:nex-wavvelet-norm equivalence}, 
\begin{align*}
\begin{split}%
    \|f_N^*\|^2_{H^s(\R^d)} &\leq A\|f_N^*\|_{\RKHSN }^2 
   \hspace{-0.1cm} = \hspace{-0.1cm}A\Big[\sum_{k\in\Z^d}\langle f^*,\phi_k\rangle^2+\sum_{j=0}^N2^{2js}\sum_{\iota\in\I}\sum_{k\in \mathbb{Z}^d} \langle f^*,\psi_{jk}^\iota\rangle^2 \big]
   \hspace{-0.05cm}\\
   &=  \hspace{-0.1cm} A\Big[\sum_{k\in\Z^d}\langle f^*,\phi_k\rangle^2+\sum_{j=0}^{N}2^{2j\so}2^{2js-2j\so}\sum_{\iota\in\I}\sum_{k\in \mathbb{Z}^d} \langle f^*,\psi_{jk}^\iota\rangle^2\Big] \\
    &\leq A2^{2N(s-\so)\vee 0}\Big[\sum_{k\in\Z^d}\langle f^*,\phi_k\rangle^2+\sum_{j=0}^{N}2^{2j\so}\sum_{\iota\in\I}\sum_{k\in \mathbb{Z}^d}\langle f^*,\psi_{jk}^\iota\rangle^2\Big]\\
    &\leq A^22^{2N(s-\so)\vee 0} \|f^*\|^2_{H^{\so}(\R^d)},
    \end{split}
\end{align*}
as desired.
\end{proof}

The next lemma gives a condition for invertibility of the kernel matrix $K_{\Phi_N}$. 

\begin{lemma}\label{lemma:ex-wavelet:PD}
Let $X_n=\{x_1,\ldots,x_n\}\subset \Omega$ be distinct points. Suppose the scaling and wavelet functions $\phi,\{\psi^\iota\}_{\iota\in\I}$ are supported on a ball $B(0,R)$ for some $R>0$. Suppose $N$ is large enough so that 
\begin{align}\label{eq:ex-wavelet-PD condition}
    2^{-N}R < \q,
\end{align}
where $\q$ is the separation distance defined in \eqref{eq:h q rho}. 
Then, the kernel matrix $K_{\Phi_N}$ is invertible. 
\end{lemma}
\begin{proof}
    The result follows from Lemma \ref{lemma:lower bound on wavelet kernel matrix} in Appendix \ref{sec:aux lemmas}, which is proved using a similar argument as in \cite[Appendix]{griebel2015multiscale}. 
\end{proof}

We are now ready to present the main result of this subsection. 
\begin{theorem}
Let $\Omega\subset \mathbb{R}^d$ satisfy Assumption \ref{assp:Omega}. Let $X_n=\{x_1,\ldots,x_n\}\subset \Omega$ be distinct points and assume the kernel matrix $K_{\Phi_N}$ is invertible. Let $f\in H^{\so}(\Omega)$. There exists $h_0$ independent of $N$ and $n$ such that if $\h\leq h_0$, then by setting $\nugg \asymp \h^{s-d/2}$ we have  
\begin{align*}
\|f - \InterN f\|_{L^2(\Omega)}&\leq C  \Big[\h^s 2^{N(s-\so)\vee 0} + \sqrt{n}\h^{\frac{d}{2}}2^{-N(s\wedge\so-\frac{d}{2})}\Big] \|f\|_{H^{\so}(\Omega)},\\
    \|f - \InterN f\|_{L^{\infty}(\Omega)}&\leq C  \Big[\h^{s-\frac{d}{2}} 2^{N(s-\so)\vee 0} + \sqrt{n}2^{-N(s\wedge\so-\frac{d}{2})} \Big] \|f\|_{H^{\so}(\Omega)},
\end{align*}
where $C$ is a constant independent of $N$ and $n$. 
If $s=s_0$ and the design points are quasi-uniform, setting $N$ so that $2^{-N}\lesssim n^{-\frac{2\so}{d(\so-d)}}$ (which also implies \eqref{eq:ex-wavelet-PD condition}) gives 
\begin{align*}
    \|f - \InterN f\|_{L^2(\Omega)}&\leq Cn^{-\frac{\so}{d}}\|f\|_{H^{\so}(\Omega)},\\
    \|f - \InterN f\|_{L^{\infty}(\Omega)}&\leq C n^{\frac12-\frac{\so}{d}}\|f\|_{H^{\so}(\Omega)},
\end{align*}
which are the best possible rates for approximating functions in $H^{\so}(\Omega)$. 
\end{theorem}

\begin{proof}
By Proposition \ref{prop:sobolev ext}, there exists $f^*\in H^{\so}(\R^d)$ such that 
\begin{align*}
    f^*|_{\Omega}=f|_{\Omega},\qquad \|f^*\|_{H^{\so}(\R^d)}\leq C \|f\|_{H^{\so}(\Omega)}
\end{align*}
for $C$ a constant independent of $f$. The first $L^2(\Omega)$  and $L^\infty(\Omega)$ bounds then follow from Theorem \ref{thm:misspecifed kriging error} part II applied to $f^*$ and the approximate function $f_N^*$ constructed in Lemma \ref{lemma:ex-wavelet:fn onen} with domain $\D=\R^d$. 

For the quasi-uniform case, recall that $\h \asymp \q \asymp n^{-1/d}.$ 
Moreover, 
\begin{align*}
    2^{-N}\lesssim n^{-\frac{2\so}{d(\so-d)}} \quad \Longrightarrow \quad \sqrt{n}2^{-N(\so-\frac{d}{2})} \lesssim \h^{\so-\frac{d}{2}} \,\,\text{ and } \,\, 2^{-N}\lesssim \q,
\end{align*}
 and so the invertibility condition \eqref{eq:ex-wavelet-PD condition} is satisfied, completing the proof. 
\end{proof}

\subsection{Finite element representations}\label{sec:FEM}

\subsubsection{Background}
In this subsection, we consider a finite element-based kernel that exploits sparsity not in the covariance matrix as in Section \ref{sec:wavelets} but instead in its inverse, the precision matrix. 
This idea was introduced in \cite{lindgren2011explicit}, which considers Mat\'ern Gaussian fields defined on a bounded domain by the SPDE \eqref{eq:SPDE bounded domain} and constructs a finite element approximation of it. 
Thanks to the compactly supported finite element bases, the resulting approximate Gaussian process enjoys a sparse precision matrix, which facilitates efficient sampling and inference. 
Since \cite{lindgren2011explicit}, finite element kernel representations have received increasing attention with various generalizations and theoretical analyses, see e.g. \cite{lindgren2022spde,bolin2020numerical,bolin2020rational,cox2020regularity,sanz2022finite}. With the exception of \cite{sanz2022finite} which studies posterior contraction, previous analyses focus on the question of kernel approximation rather than on the effect of kernel misspecification on downstream tasks such as kriging. Here, we provide the first error analysis for approximate kriging when a finite element-based kernel is employed.

To make our presentation self-contained, we first summarize the necessary background on finite elements. For simplicity, we focus on a one-dimensional example with Lagrange elements. Let $\Omega=(0,1)$ and let $\xi_i=i/N$ for $i=0,\ldots,N$. Consider the finite element space consisting of piecewise polynomials 
\begin{align*}
    V_N=\Big\{\psi\in \C^0(\overline{\Omega}): \psi \text{ is a degree-$p$ polynomial on }[\xi_{i-1},\xi_{i}], 1\leq i\leq N\Big\}
\end{align*}
for some fixed $p\geq 1$. These Lagrange elements satisfy $V_N\subset H^1(\Omega)$. Now let $\overline{V}_N$ be the space of functions obtained by subtracting the means of functions in $V_N$, i.e., 
\begin{align*}
    \overline{V}_N= \left\{\psi-\int_\Omega\psi(x)dx:\psi\in V_N\right\}. 
\end{align*}
Consider then the Galerkin approximation of the variational eigenvalue problem for $-\Delta$ over $\Omega$ with Neumann boundary condition:
\begin{align*}
    \langle \nabla \psi, \nabla \phi \rangle  = \lambda \langle \psi,\phi\rangle , \qquad \forall \phi \in \overline{V}_N,
\end{align*}
where $\langle\cdot,\cdot\rangle$ denotes the $L^2(\Omega)$ inner product. 
Notice that the Neumann boundary condition, often referred to as the natural boundary condition, does not need to be imposed explicitly. 
The reason to consider the mean-subtracted spaces $\overline{V}_N$ is that the bilinear form $a(\psi,\phi):=\langle \nabla \psi, \nabla \phi \rangle$ is coercive over $\overline{V}_N$ (see e.g. \cite[Proposition 5.3.2]{brenner2008mathematical}). 
Therefore, the above problem admits an $L^2(\Omega)$-orthonormalized eigenbasis $\{\psi_{N,i}\}_{i=1}^N$ for $\overline{V}_N$ with associated eigenvalues $\{\lambda_{N,i}\}_{i=1}^N$ satisfying (see e.g. \cite[Section 7]{boffi2010finite})
\begin{align*}
    \langle \nabla \psi_{N,i}, \nabla \phi \rangle  = \lambda_{N,i} \langle \psi_{N,i},\phi\rangle , \qquad \forall \phi \in \overline{V}_N.
\end{align*}

Now we are ready to define the finite element-based kernel. 
Let $\lambda_{N,0}=0$ and let $\psi_{N,0}$ be the constant one function so that $\{\psi_{N,i}\}_{i=0}^N$ is an $L^2(\Omega)$-orthonormalized basis for $V_N$. Consider the kernel
\begin{align}\label{eq:ex-FEM-kernel}
    \Phi_N(x,\tilde{x})=\sum_{i=0}^{N} \left(1+\lambda_{N,i}\right)^{-1}\psi_{N,i}(x)\psi_{N,i}(\tilde{x}).
\end{align}
The RKHS generated by $\Phi_N$ is 
\begin{align*}
    \RKHSN  &= \left\{g=\sum_{i=0}^N w_i \psi_{N,i}: \|g\|_{\RKHSN }^2:=\sum_{i=0}^N w_i^2 (1+\lambda_{N,i})<\infty\right\}
    =\text{span}\{\psi_{N,0},\ldots,\psi_{N,N}\},
\end{align*}
which as a vector space agrees with the space $V_N\subset H^1(\Omega)$. We notice further that the RKHS norm $\|\cdot\|_{\RKHSN }$ coincides with the $H^1(\Omega)$ norm over $\RKHSN $. Indeed, for $g\in \RKHSN $ we have 
\begin{align*}
    \|\nabla g\|_{L^2(\Omega)}^2 = \Big\|\sum_{i=0}^N w_i \nabla \psi_{N,i}\Big\|_{L^2(\Omega)}^2 = \sum_{i=0}^N \sum_{j=0}^N w_iw_j \langle \nabla \psi_{N,i},\nabla \psi_{N,j} \rangle =\sum_{i=0}^N w_i^2 \lambda_{N,i}, 
\end{align*}
so that 
\begin{align}\label{eq:EX-FEM:norm equiv}
    \|g\|_{H^1(\Omega)}^2 = \|g\|_{L^2(\Omega)}^2 + \|\nabla g\|_{L^2(\Omega)}^2 = \sum_{i=0}^N w_i^2 (1+\lambda_{N,i}) = \|g\|_{\RKHSN }^2. 
\end{align}

\subsubsection{Error analysis}
We shall apply Theorem \ref{thm:misspecifed kriging error} part II for the error analysis. The following lemma establishes the approximation error of a Sobolev function by functions in $V_N$.

\begin{lemma}\label{lemma:ex-FEM-fN OneN}
Let $f\in H^{\so}(\Omega)$ with $1\leq s_0\in \mathbb{N}$. There exists $f_N\in V_N$ such that 
\begin{align*}
    \|f-f_N\|_{L^{\infty}(\Omega)} &\leq C N^{-\big[(p+1)\wedge s_0-\frac12\big]}\|f\|_{H^{\so}(\Omega)},\\
    \|f_N\|_{H^1(\Omega)} &\leq C \|f\|_{H^{\so}(\Omega)},
\end{align*}
where $C$ is a constant independent of $N$. 
\end{lemma}
\begin{proof}
The approximation $f_N$ will be taken as the finite element interpolation of $f$. 
Recall that the finite element space $V_N$ consists of piecewise polynomials of degree $p$ and the associated nodal variables only involve function values but not derivatives, so that the assumption of \cite[Theorem 4.4.4]{brenner2008mathematical} is satisfied with $m=p+1$ and $\ell=0$.
Then, letting $f_N$ be the finite element interpolant in \cite[Theorem 4.4.20]{brenner2008mathematical}, we have for $f\in H^{\so}(\Omega)\subset H^{(p+1)\wedge \so}(\Omega)$ and $h=N^{-1}$ that 
\begin{align*}
    \|f-f_N\|_\infty & \leq C N^{-\big[(p+1)\wedge \so-\frac12\big]} \|f\|_{H^{(p+1)\wedge \so}(\Omega)}\leq C N^{-\big[(p+1)\wedge \so-\frac12\big]} \|f\|_{H^{\so}(\Omega)},\\
    \|f-f_N\|_{H^1(\Omega)}& \leq C N^{-\big[(p+1)\wedge \so-1\big]} \|f\|_{H^{(p+1)\wedge \so}(\Omega)}\leq C N^{-\big[(p+1)\wedge \so-1\big]} \|f\|_{H^{\so}(\Omega)},
\end{align*}
where $C$ is a constant independent of $N$. 
The second assertion further implies that 
\begin{align*}
    \|f_N\|_{H^1(\Omega)} \leq C \|f\|_{H^1(\Omega)} + C N^{-\big[(p+1)\wedge \so-1\big]} \|f\|_{H^{\so}(\Omega)} \leq C\|f\|_{H^{\so}(\Omega)},
\end{align*}
where we have used that $(p+1)\wedge s_0\geq 1$.

\end{proof}

The next lemma gives a simple sufficient condition for invertibility of the kernel matrix $K_{\Phi_N}$. 
\begin{lemma}\label{lemma:ex-FEM-PD}
Let $X_n=\{x_1,\ldots,x_n\}\subset \Omega$ be distinct points. 
Suppose the separation distance $\q$ defined in \eqref{eq:h q rho} satisfies 
\begin{align}\label{eq:ex-FEM-PD condition}
    \q > N^{-1}.
\end{align}
Then, the kernel matrix $K_{\Phi_N}$ is invertible.  
\end{lemma}
\begin{proof}
    By Lemma \ref{lemma:positive definiteness} it suffices to show the existence of functions $\{f_j\}_{j=1}^n\subset V_N$ such that $f_j(x_i)=\delta_{ij}$. For each $x_i$, let $T=[\xi_j,\xi_{j+2}]$ be the interval that contains it so that $x_i$ is not one of the boundary nodes. 
    Since over each subinterval $[\xi_j,\xi_{j+1}]$ and $[\xi_{j+1},\xi_{j+2}]$ the finite element space $V_N$ consists of all polynomials of degree $p$, there exists a function $\psi^{(i)}\in V_N$ supported over $T$ such that $\psi^{(i)}(x_i)\neq 0$. 
    The assumption \eqref{eq:ex-FEM-PD condition} implies that $\operatorname{min}_{j\neq i}|x_j-x_i|> 2N^{-1}$ so that any other points $x_j$ for $j\neq i$ must lie outside of $T$ since the length of $T$ is $2N^{-1}$. 
    Therefore, $f_i=\psi^{(i)}/\psi^{(i)}(x_i)$ satisfies the requirements. 
\end{proof}

Now we are ready to state the main result of this subsection. 
\begin{theorem}
Let $\Omega=(0,1)$ and $X_n=\{x_1,\ldots,x_n\}\subset \Omega$ be distinct points. Suppose the kernel matrix $K_{\Phi_N}$ is invertible. 
Let $f\in H^{\so}(\Omega)$ with $1\leq s_0 \in \mathbb{N}$. Then, there exists $h_0$ independent of $N$ and $n$ such that if $\h\leq h_0$, we have by setting $\sqrt{\nugg}\asymp \sqrt{\h}$ that 
\begin{align*}
        \|f- \InterN f\|_{L^2(\Omega)}& \leq C\Big(\sqrt{n\h}N^{-\big[(p+1)\wedge \so-\frac12\big]}+\h\Big)\|f\|_{H^{\so}(\Omega)},\\
         \|f- \InterN f\|_{L^{\infty}(\Omega)}& \leq C\Big(\sqrt{n}N^{-\big[(p+1)\wedge \so-\frac12\big]}+\sqrt{\h}\Big)\|f\|_{H^{\so}(\Omega)}, 
\end{align*}
where $C$ is a constant independent of $N$ and $n$. 

When $s_0=1$, if the design points are quasi-uniform, setting $N\gtrsim n$ to guarantee invertibility of $K_{\Phi_N},$ we have 
\begin{align*}
    \|f- \InterN f\|_{L^2(\Omega)}& \leq Cn^{-1}\|f\|_{H^1(\Omega)},\\
         \|f- \InterN f\|_{L^{\infty}(\Omega)}& \leq Cn^{-1/2}\|f\|_{H^1(\Omega)}, 
\end{align*}
which are the best possible rates for approximating functions in $H^1(\Omega)$.

\end{theorem}

\begin{proof}
The first $L^2$ and $L^\infty$ bounds follow from Theorem \ref{thm:misspecifed kriging error} part II with the approximation $f_N$ constructed in Lemma \ref{lemma:ex-FEM-fN OneN} with $s_N\equiv 1$ and $\D=\Omega$. For the quasi-uniform case, we set $N$ to balance the two error terms, which gives $N\gtrsim n^{\frac{1}{p+0.5}}$. By \eqref{eq:ex-FEM-PD condition}, $N \gtrsim n$ ensures invertibility since by quasi-uniformness $\q\asymp n^{-1}$.
\end{proof}

\section{Conclusions and discussion}\label{sec:conclusions}
This paper has introduced a unified framework to analyze Gaussian process regression under computational and epistemic misspecification. We have illustrated this framework in a wide range of examples, recovering existing results and obtaining new ones for simultaneous epistemic and computational misspecification in Karhunen-Lo\`eve expansions, wavelet-based multiscale kernels, and finite element kernels. We believe that our framework could be applied to analyze other kernel methods based, for instance, on hierarchical matrices or graphs. Another important direction for future research is to leverage our new theory for Gaussian process regression to analyze downstream tasks, including Bayesian optimization and posterior sampling with surrogate likelihood functions.

\section*{Acknowledgments}
This work was partially funded by NSF CAREER award NSF DMS-2237628, DOE DE-SC0022232, and the BBVA Foundation.
The authors are thankful to Omar Al-Ghattas and Jiaheng Chen for providing generous feedback on an earlier version of this manuscript.

\appendix

\section{Proofs of Lemmas in Section \ref{sec:proof of main result}}\label{sec:Appendix1}

\begin{proof}[Proof of Lemma \ref{lemma:sampling theorem}]
The first assertions are special cases of \cite[Theorem 3.1]{arcangeli2012extension}, corresponding to parameters $(p,q,\varkappa)=(2,2,2)$ and $(p,q,\varkappa)=(2,\infty,2)$ respectively. 
The fact that $h_{\lfloor s\rfloor,d,\Omega}$ depends only on the integer part of $s$ is explained in their Remark 3.2. 
When $u|_{X_n}=0$, we can employ the sampling theorem from \cite[Theorem 2.12]{narcowich2005sobolev}, where the new constant $C_{\lfloor s\rfloor,d,\Omega}$ depends only on the integer part of $s$. 
\end{proof}

\begin{proof}[Proof of Lemma \ref{lemma:regularized LS}]
First, for $\lambda>0$ the approximate interpolant $\Interpsi^\nugg f$ is the unique solution to the regularized least-squares problem (see e.g. \cite[Theorem 3.4]{kanagawa2018gaussian})
    \begin{align}\label{eq:regularized LS}
        \Interpsi^\nugg f=\underset{\widehat{f}\in \RKHSpsi}{\operatorname{arg\,min}}\,\, \sum_{i=1}^n | f(x_i) - \widehat{f}(x_i)|^2 +\nugg \|\widehat{f}\|_{\RKHSpsi}^2.
    \end{align}
On the other hand, the interpolant $\Interpsi f$ defined in \eqref{eq:kriging estimator} is the minimum-norm interpolant in $\RKHSpsi$ of the function values $f(x_1),\ldots,f(x_n)$ (see e.g. \cite[Theorem 3.5]{kanagawa2018gaussian})
\begin{align*}
    \Interpsi f = \underset{\widehat{f}\in \RKHSpsi}{\operatorname{arg\,min}}\,\,  \|\widehat{f}\|_{\RKHSpsi} \quad \text{ subject to   }  \quad \widehat{f}(x_i)=f(x_i) \quad \forall i=1,\ldots,n.
\end{align*}
In particular, $\Interpsi f(x_i)=f(x_i)$ for all $i$ and $\|\Interpsi f\|_{\RKHSpsi}\leq \|f\|_{\RKHSpsi}$. This implies by \eqref{eq:regularized LS} that 
\begin{align*}
    \|f - \Interpsi^\nugg f\|^2_{\dtwo} \vee \nugg\|\Interpsi^\nugg f\|^2_{\RKHSpsi}
    &\leq \sum_{i=1}^n |f(x_i) - \Interpsi^\nugg f(x_i)|^2+\nugg \|\Interpsi^\nugg f\|^2_{\RKHSpsi} \\
    &\leq \sum_{i=1}^n |f(x_i) - \Interpsi f(x_i)|^2 +\nugg \|\Interpsi f\|^2_{\RKHSpsi}\\
    & = \nugg \|\Interpsi f\|^2_{\RKHSpsi}\leq \nugg \|f\|_{\RKHSpsi}^2,
\end{align*}
yielding the claims. 
\end{proof}

\begin{proof}[Proof of Lemma \ref{lemma:approx inter error on Xn}]
   We let $Y:=\bigl(g(x_1),\ldots,g(x_n)\bigr)^T$ and write
   \begin{align*}   
    \begin{bmatrix}
        \Interpsi^{\nugg} g(x_1) \\
        \vdots\\
        \Interpsi^{\nugg} g(x_n)
    \end{bmatrix}   
    =
    \begin{bmatrix}
        k_{\Psi}(x_1)^T\\
        \vdots\\
        k_{\Psi}(x_n)^T      
    \end{bmatrix}
    (K_{\Psi}+\nugg I)^{-1} Y &= K_{\Psi}(K_{\Psi}+\nugg I)^{-1}Y. 
   \end{align*}
   Therefore, denoting by $\|\cdot\|_2$ the vector 2-norm,
   \begin{align*}
       \|\Interpsi^{\nugg}g -g\|_{\dtwo} &= \|K_{\Psi}(K_{\Psi}+\nugg I)^{-1}Y-Y\|_2 \\
       &= \nugg \|(K_{\Psi}+\nugg I)^{-1}Y\|_2
       \leq \frac{\nugg}{\smin(K_{\Psi})+\nugg} \|Y\|_2.
   \end{align*}
  The first claim follows by noticing that $\|Y\|_2= \|g\|_{\dtwo}$. To show the second claim, let $\alpha = (K_{\Psi}+\nugg I)^{-1}Y$. Then, $\Interpsi^\nugg g=$ $\sum_{i=1}^n \alpha_i \Psi(\cdot,x_i)$ and 
\begin{align*}
    \|\Interpsi^\nugg g\|_{\RKHSpsi}^2
    &= \alpha^T K_{\Psi} \alpha 
    = Y^T(K_{\Psi}+\nugg I)^{-1} K_{\Psi}(K_{\Psi}+\nugg I)^{-1}Y\\
    & \leq 
    Y^T(K_{\Psi}+\nugg I)^{-1} Y
     \leq \frac{\|Y\|_2^2}{\smin(K_{\Psi})+\nugg} . 
\end{align*}
\end{proof}

\section{Auxiliary Results for Section \ref{sec:examples}} \label{sec:aux lemmas}
\begin{proposition}\label{prop:sobolev ext}
    Let $\Omega\subset \R^d$ be a bounded domain with Lipschitz boundary. Let $s>\frac{d}{2}$, there exists a bounded extension operator $E_{\Omega}:H^s(\Omega) \rightarrow H^s(\mathbb{R}^d)$: for all $g\in H^s(\Omega)$, $E_{\Omega}g=g$ over $\Omega$ and
    \begin{align*}
        \|E_{\Omega}g\|_{H^s(\mathbb{R}^d)} \leq C_{s,d,\Omega} \|g\|_{H^s(\Omega)}.
    \end{align*}
   
\end{proposition}
\begin{proof}
    The case when $s$ is an integer is stated in \cite[Theorem 1.4.5]{brenner2008mathematical}. The general case is discussed in \cite[p. 374]{brenner2008mathematical}. 
\end{proof}

\begin{proposition}[\cite{narcowich2006sobolev}, Proposition 3.1]\label{prop:exist interpolant abstract}
    Let $\mathcal{Y}$ be a Banach space, $\mathcal{V}\subset \mathcal{Y}$ a subspace, and $\mathcal{Z}'$ a finite-dimensional subspace of the dual space $\mathcal{Y}'$. Suppose that, for every $z'\in \mathcal{Z}'$ and some $\hat{\gamma}>1$ independent of $z',$ it holds that
    \begin{align*}
        \|z'\|_{\mathcal{Y}'} \leq \hat{\gamma} \|z'|_{\mathcal{V}}\|_{\mathcal{V}'}.
    \end{align*}
    Then, for any $y\in\mathcal{Y},$ there exists $v=v(y)\in\mathcal{V}$ such that $v(y)$ interpolates $y$ on $\mathcal{Z}'$, i.e., $z'(y)=z'(v)$ for all $z'\in\mathcal{Z}'$. In addition, $v$ approximates $y$ in the sense that 
    \begin{align*}
        \|v-y\|_{\mathcal{Y}} \leq (1+2\hat{\gamma}) \operatorname{dist}_{\mathcal{Y}}(y,\mathcal{V}). 
    \end{align*}
\end{proposition}

\begin{proposition}\label{prop:exist interpolant mercer kernel}
    Let $\Omega\subset \mathbb{R}^d$ be a domain and $I$ be a countable index set. Suppose $\{\psi_{ij}\}_{i\in \mathbb{N},j\in I}$ is an orthonormal set in $L^2(\Omega)$ and $\{\lambda_i\}_{i\in\mathbb{N}}$ is a sequence of nonnegative real numbers satisfying
    \begin{align}\label{eq:basis summability}
        \underset{x\in \Omega}{\operatorname{sup}}\, \,\sum_{j\in I} \psi_{ij}(x)^2 \leq A_i,\qquad \sum_{i=1}^{\infty} \lambda_i A_i<\infty
    \end{align}
    for some sequence of positive numbers $\{A_i\}_{i=1}^{\infty}$. 
    Consider the kernels
    \begin{align*}
        \Phi_N(x,\tilde{x})=\sum_{i=1}^N \lambda_i\sum_{j\in I} \psi_{ij}(x)\psi_{ij}(\tilde{x}),\qquad \Phi(x,\tilde{x})=\sum_{i=1}^{\infty} \lambda_i\sum_{j\in I} \psi_{ij}(x)\psi_{ij}(\tilde{x}).
    \end{align*}
Suppose $N$ is large enough so that 
    \begin{align}\label{eq:1/gamma bound}
        \sum_{i=N+1}^{\infty}\lambda_i A_i \leq \frac{\sigma_{\min}(K_{\Phi})}{n\gamma^2},
    \end{align}
    where $\gamma>1$ is a fixed constant and ${\sigma_{\min}(K_{\Phi})}$ is the smallest eigenvalue of $K_\Phi.$
    Then, for any $f\in \RKHS$, there exists $f_N\in\RKHSN $ such that 
    \begin{align*}
        f|_{X_n}=f_N|_{X_n} ,\qquad \|f-f_N\|_{\RKHS}\leq (1+2\hat{\gamma}) \operatorname{dist}_{\RKHS} (f,\RKHSN ),
    \end{align*}
    where $\hat{\gamma}=\gamma(\gamma-1)^{-1}$ and $ \operatorname{dist}_{\RKHS} (f,\RKHSN )=\operatorname{inf}_{g\in \RKHSN }\|f-g\|_{\RKHS}$, which in particular implies that 
    \begin{align*}
        \|f_N\|_{\RKHS} \leq (2+2\hat{\gamma}) \|f\|_{\RKHS}. 
    \end{align*}
\end{proposition}
\begin{proof}
    The RKHSs associated with $\Phi_N$ and $\Phi$ take the form 
    \begin{align*}
        \RKHSN  &= \left\{g=\sum_{i=1}^N \sum_{j\in I}a_{ij} \psi_{ij}: \|g\|^2_{\RKHSN }:=\sum_{i=1}^N \lambda_i^{-1}\sum_{j\in I}a_{ij}^2 <\infty\right\},\\
        \RKHS &= \left\{g=\sum_{i=1}^{\infty} \sum_{j\in I}a_{ij} \psi_{ij}: \|g\|^2_{\RKHS}:=\sum_{i=1}^{\infty} \lambda_i^{-1} \sum_{j\in I} a_{ij}^2 <\infty\right\}. 
    \end{align*}
    We shall apply Proposition \ref{prop:exist interpolant abstract} with $\mathcal{Y}=\RKHS$, $\mathcal{V}=\RKHSN $ and $\mathcal{Z}'=$span$\{\delta_{x_1},\ldots,\delta_{x_n}\}$. Now take $z'=\sum_{k=1}^n z_k\delta_{x_k}\in \mathcal{Z}'$. We need to show that, for some constant $\hat{\gamma},$ 
    \begin{align*}
        \|z'\|_{\RKHS'} \leq \hat{\gamma} \|z'|_{\RKHSN }\|_{\RKHSN '}.
    \end{align*}
    By Riesz's representation theorem and the reproducing property, we have 
    \begin{align}\label{eq:z prime H norm}
        \|z'\|_{\RKHS'}=\Big\|\sum_{k=1}^n z_k\delta_{x_k}\Big\|_{\RKHS'}= \Big\|\sum_{k=1}^n z_k \Phi(\cdot,x_k)\Big\|_{\RKHS}= : \|v\|_{\RKHS}.
    \end{align}
    To compute the $\RKHSN '$ norm of $z'|_{\RKHSN }$, take $g=\sum_{i=1}^N\sum_{j\in I} a_{ij} \psi_{ij}\in \RKHSN $. Then, 
    \begin{align}\label{eq:z prime g}
        z'(g)= \sum_{k=1}^n z_k \sum_{i=1}^N\sum_{j\in I}  a_{ij} \psi_{ij}(x_k)=\sum_{i=1}^N \sum_{j\in I}a_{ij} \sum_{k=1}^n z_k \psi_{ij}(x_k)= :\langle g, v_N \rangle_{\RKHSN }, 
    \end{align}
    where 
    \begin{align*}
        v_N&=\sum_{i=1}^N \lambda_i\sum_{j\in I}\Bigl(\sum_{k=1}^n z_k\psi_{ij}(x_k) \Bigr)\psi_{ij}(\cdot)\\
        & = \sum_{k=1}^n z_k\Bigl(\sum_{i=1}^N\lambda_i \sum_{j\in I} \psi_{ij}(x_k)\psi_{ij}(\cdot)\Bigr)
         =\sum_{k=1}^n z_k \Phi_N(\cdot,x_k).   
    \end{align*}
    Now by Cauchy-Schwarz, \eqref{eq:z prime g} implies that 
    \begin{align*}
        |z'(g)| \leq \|g\|_{\RKHSN } \|v_N\|_{\RKHSN },
    \end{align*}
    so that $\|z'|_{\RKHSN }\|_{\RKHSN '}\leq \|v_N\|_{\RKHSN }$. Taking $g=v_N$ then gives that
    \begin{align} \label{eq:z prime HN norm}
        \|z'|_{\RKHSN }\|_{\RKHSN '}= \|v_N\|_{\RKHSN }. 
    \end{align}
    Combining \eqref{eq:z prime H norm} and \eqref{eq:z prime HN norm}, the result will follow if we can show 
    \begin{align*}
        \|v\|_{\RKHS}\leq \hat{\gamma} \|v_N\|_{\RKHSN }
    \end{align*}
    for some constant $\hat{\gamma}$. To see this, we first notice that $\|\cdot \|_{\RKHS}$ coincides with $\|\cdot \|_{\RKHSN }$ over $\RKHSN $, and compute 
    \begin{align*}
        \|v-v_N\|_{\RKHS}^2 & \hspace{-0.1cm}= \hspace{-0.05cm} \Big\|\sum_{k=1}^n z_k \bigl(\Phi(\cdot,x_k)-\Phi_N(\cdot,x_k)\bigr)\Big\|_{\RKHS}^2
       \hspace{-0.24cm}  = \hspace{-0.05cm}\Bigl\|\sum_{k=1}^n z_k \sum_{i=N+1}^{\infty} \hspace{-0.15cm} \lambda_i\sum_{j\in I}\psi_{ij}(x_k)\psi_{ij}(\cdot) \Bigr\|_{\RKHS}^2 \\
         & \hspace{-0.1cm} = \Bigl\|\sum_{i=N+1}^{\infty}\sum_{j\in I} \lambda_i\Bigl(\sum_{k=1}^n z_k\psi_{ij}(x_k)\Bigr) \psi_{ij}(\cdot)\Bigr\|_{\RKHS}^2
        \hspace{-0.15cm}  = \hspace{-0.05cm} \sum_{i=N+1}^{\infty} \hspace{-0.15cm}  \lambda_i \sum_{j\in I} \Bigl(\sum_{k=1}^nz_k\psi_{ij}(x_k)\Bigr)^2.
    \end{align*}
    By Cauchy-Schwarz applied to $\sum_{k=1}^n z_k \psi_{ij}(x_k)$, the last bound together with \eqref{eq:basis summability} implies that
    \begin{align*}
        \|v-v_N\|_{\RKHS}^2&
        \leq \sum_{i=N+1}^{\infty} \lambda_i\sum_{j\in I}\sum_{k=1}^n z_k^2 \sum_{k=1}^n \psi_{ij}(x_k)^2 \\
        &\leq \sum_{i=N+1}^{\infty} \lambda_i\sum_{k=1}^n z_k^2 \sum_{k=1}^n \sum_{j\in I}\psi_{ij}(x_k)^2 \leq n 
        \sum_{k=1}^n z_k^2 \sum_{i=N+1}^{\infty}\lambda_iA_i.  
    \end{align*}
    On the other hand, \eqref{eq:z prime H norm} implies that 
    \begin{align*}
        \|v\|_{\RKHS}^2 =\sum_{i=1}^n \sum_{j=1}^n z_iz_j \Phi(x_i,x_j) \geq \sigma_{\text{min}}(K_{\Phi}) \sum_{k=1}^nz_k^2.
    \end{align*}
   Therefore, using \eqref{eq:1/gamma bound},
    \begin{align*}
        \|v-v_N\|_{\RKHS} \leq \sqrt{\frac{n\sum_{i=N+1}^{\infty}\lambda_i A_i}{\sigma_{\text{min}}(K_{\Phi})}} \|v\|_{\RKHS}\leq \frac{1}{\gamma} \|v\|_{\RKHS}.
    \end{align*}
    Hence, 
    \begin{align*}
        \|v\|_{\RKHS} \leq \|v-v_N\|_{\RKHS} +\|v_N\|_{\RKHS}\leq \frac{1}{\gamma} \|v\|_{\RKHS}+\|v_N\|_{\RKHSN },
    \end{align*}
    where we have used that $\|\cdot\|_{\RKHS}$ coincides with $\|\cdot\|_{\RKHSN }$ on $\RKHSN $. Rearranging, 
    \begin{align*}
        \|v\|_{\RKHS} \leq \frac{\gamma}{\gamma-1} \|v_N\|_{\RKHSN }=\hat{\gamma} \|v_N\|_{\RKHSN }, 
    \end{align*}
    which concludes the proof. 
\end{proof}

\begin{lemma}\label{lemma:positive definiteness}
    Let $\Phi_N(x,\tilde{x})=\sum_{i=1}^N\lambda_i \psi_i(x)\psi_i(\tilde{x})$. Suppose $\{x_1,\ldots,x_n\}\subset \Omega$ are $n$ distinct points and let $K\in\mathbb{R}^{n\times n}$ be the covariance matrix whose entries are $K_{ij}=\Phi_N(x_i,x_j)$. Then, $K$ is invertible if there exist functions $\{f_i\}_{i=1}^n\subset \operatorname{span}\{\psi_1,\ldots,\psi_N\}$ such that $f_i(x_j)=\delta_{ij}$.
\end{lemma}
\begin{proof}
    The given conditions imply that the point evaluation functionals $\{\delta_{x_i}\}_{i=1}^n$ are linearly independent. Indeed, suppose $\sum_{i=1}^n a_i \delta_{x_i}=0$. Then, 
    \begin{align*}
        0=\Big(\sum_{i=1}^n a_i \delta_{x_i} \Big)(f_j)= \sum_{i=1}^n a_i f_j(x_i)=a_j,\qquad j=1,\ldots,n.   
    \end{align*}
    Denoting $\RKHSN '$  the dual space of $\RKHSN $, it follows by Riesz representation and the reproducing property of $\Phi_N$ that $\langle \delta_{x_i},\delta_{x_j}\rangle_{\RKHSN '}=\Phi_N(x_i,x_j)$. Therefore, the covariance matrix $K$ is the Gram matrix associated with $\{\delta_{x_i}\}_{i=1}^n$, whose linear independence is equivalent to positive definiteness of $K$. 
\end{proof}

\begin{lemma}[{{\cite[Appendix]{griebel2015multiscale}}}]\label{lemma:lower bound on wavelet kernel matrix}
Suppose $2^{-N}R<\q$. Then, the kernel matrix $K_{\Phi_N}=\{\Phi_N(x_i,x_j)\}_{i,j=1}^n$ with $\Phi_N$ defined in \eqref{eq:wavelet kernel} is invertible. 
\end{lemma}
\begin{proof}
    First, we notice that the kernel matrix can be written as 
    \begin{align*}
        K_{\Phi_N} = K_{-1}+\sum_{j=0}^N 2^{-2js} K_j,
    \end{align*}
    where 
    \begin{align*}
        [K_{-1}]_{pq}&=\sum_{k\in \Z^d} \phi_k(x_p)\phi_k(x_q),\\
        [K_j]_{pq}&=\sum_{\iota\in \I}\sum_{k\in \mathbb{Z}^d} \psi^\iota_{jk}(x_p)\psi^\iota_{jk}(x_q)\qquad j\geq 0.
    \end{align*}
    Therefore, for $u\in \R^n$, 
    \begin{align*}
        u^T K_{\Phi_N} u =u^TK_{-1}u+ \sum_{j=0}^N 2^{-2js} u^TK_j u \geq 2^{-2Ns} u^T K_N u \geq 2^{-2Ns} \sigma_{\min}(K_N) u^Tu,
    \end{align*}
    which implies that 
    \begin{align}\label{eq:smallest eval of wavelet kernel matrix}
        \sigma_{\min}(K_{\Phi_N})\geq 2^{-2Ns} \sigma_{\min}(K_N).
    \end{align}         
    Recall that the wavelet functions $\{\psi^\iota\}_{\iota\in\I}$ is supported in $B(0,R)$ so that each $\psi_{jk}^\iota$ has support size $2^{-j}R$. Now the assumption $2^{-N}R < \q$ implies that for any $p\neq q$, 
    \begin{align*}
        [K_N]_{pq}=\sum_{\iota\in\I}\sum_{k\in \mathbb{Z}^d} \psi_{Jk}^\iota(x_p)\psi_{Jk}^\iota(x_q)=0
    \end{align*}
    since only one of $x_p$ or $x_q$ could lie in the support of the same $\psi_{Jk}^\iota$. Therefore, $K_N$ is diagonal with entries
    \begin{align*}
        [K_N]_{pp}=\sum_{\iota\in\I}\sum_{k\in \mathbb{Z}^d} \psi_{Jk}^\iota(x_p)^2 \geq  2^{Nd} \underset{x\in \Omega}{\operatorname{inf}}\,\,\sum_{\iota\in\I}\sum_{k\in \mathbb{Z}^d} \psi^\iota(2^Nx-k)^2\geq c2^{Nd},\qquad \forall p
    \end{align*}
    for a constant $c>0$ (which can be determined numerically for e.g. the Daubechies wavelet). 
    Hence $\smin(K_N)>0$, which together with \eqref{eq:smallest eval of wavelet kernel matrix} gives invertibility of $K_{\Phi_N}$.  
\end{proof}

\begin{lemma}[{{\cite[Theorem 4.4.20]{brenner2008mathematical}}}]
\label{lemma:FEM approx error}
Suppose $\Omega$ is a polyhedral domain in $\mathbb{R}^d$ and  let $\{\mathcal{T}^h\}$ be a family of non-degenerate  triangulation. Consider a reference finite element $(K,\mathcal{P},\mathcal{N})$ satisfying 
\begin{enumerate}
    \item $K$ is star-shaped with respect to some ball;
    \item $\mathcal{P}_{m-1}(K)\subset \mathcal{P}\subset W^{m}_{\infty}(K)$ with $\mathcal{P}_m(K)$ the set of polynomials of degree less than or equal to $m$ over $K$;
    \item $\mathcal{N}\subset (C^{\ell}(K))'$, i.e., the nodal variables include derivatives up to order $\ell$. 
\end{enumerate}
For all $T\in\mathcal{T}^h$, let $(T,\mathcal{P}_T,\mathcal{N}_T)$ be the affine-equivalent elements. Suppose $m>\ell+d/2$. 
Then, there exists a constant $C$ independent of $v$ such that, for $0\leq r\leq m,$
\begin{align*}
    \Big(\sum_{T\in \mathcal{T}^h}\|v-\Pi_h v\|_{W_2^r(T)}\Big)^{1/2} &\leq C h^{m-r} |v|_{W_2^m(\Omega)}, \quad \forall v\in W_2^m(\Omega),
\end{align*}
where $\Pi_h$ is the finite element interpolation operator. Moreover, for $0\leq r \leq \ell,$ 
\begin{align*}
    \underset{T\in \mathcal{T}^h}{\operatorname{max}}\,\,\|v-\Pi_h v\|_{W_{\infty}^{r}(T)} & \leq C  h^{m-r-d/2} |v|_{W_2^m(\Omega)}, \quad \forall v\in W_2^m(\Omega).
\end{align*}
\end{lemma}

\bibliographystyle{abbrvnat} 
\bibliography{bib}       

\end{document}